\documentclass[reqno]{amsart}
\usepackage{hyperref,xcolor,amssymb}
\usepackage[margin=1in]{geometry}
\usepackage[T1]{fontenc}
\usepackage[utf8]{inputenc}

\DeclareMathOperator*{\esssup}{ess\,sup}
\def\loc{\mathrm{loc}}
\DeclareMathOperator*{\diam}{diam}

\DeclareMathOperator*{\data}{data}

\newtheorem{lemma}{Lemma}
\newtheorem{remark}{Remark}
\newtheorem{definition}{Definition}
\newtheorem{theorem}{Theorem}
\newtheorem{corollary}{Corollary}
\numberwithin{equation}{section}

\DeclareMathOperator*{\essinf}{ess\,inf}
\newcommand{\dd}{\mathrm{d}}
\newcommand{\osc}{\mathrm{osc}}

\def\Xint#1{\mathchoice
	{\XXint\displaystyle\textstyle{#1}}%
	{\XXint\textstyle\scriptstyle{#1}}%
	{\XXint\scriptstyle\scriptscriptstyle{#1}}%
	{\XXint\scriptscriptstyle\scriptscriptstyle{#1}}%
	\!\int}
\def\XXint#1#2#3{{\setbox0=\hbox{$#1{#2#3}{\int}$ }
		\vcenter{\hbox{$#2#3$ }}\kern-.6\wd0}}

\def\dashint{\Xint-}

\title{Regularity results for quasiminima of a class of double phase problems}
\author{Antonella Nastasi, Cintia Pacchiano Camacho}
\address{Antonella Nastasi \newline
University of Palermo, Department of Engineering, Viale delle Scienze, 90128 Palermo, Italy}
\email{antonella.nastasi@unipa.it}

\address{Cintia Pacchiano Camacho\newline
	 University of Calgary,  Department of Mathematics and Statistics, Calgary, Canada}
\email{cintia.pacchiano@ucalgary.ca}

\begin{document}

\begin{abstract}
  We prove boundedness, H\"older continuity, Harnack inequality results for local quasiminima to elliptic double phase problems of $p$-Laplace type in the general context of metric measure spaces. The proofs follow a variational approach and they are based on the De Giorgi method, a careful phase analysis and estimates in the intrinsic geometries.

\end{abstract}

  \maketitle
  
\section{Introduction}

We consider a complete metric measure space $(X,d,\mu)$ endowed with a metric $d$ and a doubling measure $\mu$ and supporting a weak $(1,p)$-Poincar\'e inequality.
Let  $\Omega$ be an open subset of $X$. 
This paper discusses regularity properties of local quasiminima of the double phase integral 
\begin{equation}\label{J}
\int_{\Omega}H(x,g_{u}) \,\dd \mu
=\int_{\Omega}(g_{u}^p + a(x) g_{u}^q) \,\dd \mu,
\end{equation} 
where $g_u$ is the minimal $p$-weak upper gradient of $u$ (see Definition \ref{pweakuppgrad}), $a(\cdot)\geq 0$ and
\begin{equation}\label{pqcond}
1\leq\frac{q}{p}\le1+\frac{\alpha}{Q}, 
\quad p>1,
\quad 0<\alpha\le1, 
\quad Q=\log_2C_D
\end{equation} 
with $C_D$ doubling constant of the measure. 
Note that $Q$ is a notion of dimension related to the measure $\mu$. 
For example, in the Euclidean $n$-space with the Lebesgue measure we have $Q=n$.
The double phase functional in \eqref{J} is denoted by
\[
H(x,z)= |z|^p+a(x)|z|^q, 
\quad x \in \Omega,\quad z\in \mathbb{R}.
\]
The modulating coefficient function $a(\cdot)\geq 0$ will be assumed to satisfy some standard regularity conditions, see \eqref{aalpha} below for the precise hypotheses set.
The integral \eqref{J} is characterised by an energy density switching between two kind of degenerate behaviours, based on the size of the modulating coefficient denoted as $a(\cdot)$, which determines the phase. When $a(x)=0$, the variational integral in \eqref{J} reduces to the familiar problem with $p$-growth and when $a(x)\ge c>0$ we have the $(p,q)$-problem. We note that the natural condition \eqref{pqcond} on the gap between the exponents $p$ and $q$ for the regularity of
quasiminima is sharp.

Starting from the seminal papers of Marcellini \cite{Ma1, Ma2, Ma3}, the regularity theory for minimizers of double phase functionals has been treated at length in the euclidean setting. Recently Colombo-Mingione \cite{CM1,CM}, Cupini-Marcellini-Mascolo \cite{CMM}, De Filippis-Mingione \cite{DFM}, Esposito-Leonetti-Mingione \cite{ELM}, Di Marco-Marcellini \cite{DMM}, Mingione \cite{Min}, Tachikawa \cite{Tach} considered classes of nonuniformly elliptic problems and proved different regularity properties. For example, H\"{o}lder regularity results for certain class of double phase problems with non standard growth conditions can be found in Ciani-Skrypnik-Vespri \cite{CSV}, D\"{u}zg\"{u}n-Marcellini-Vespri \cite{DMV1, DMV2}, Eleuteri \cite{Ele}, Harjulehto-H\"{a}st\"{o}-Toivanen \cite{HHT}, Di Benedetto-Gianazza-Vespri \cite{DGV}, Liao-Skrypnik-Vespri \cite{LSV} and Liskevich–Skrypnik \cite{LS}. Harnack inequalities for double phase problems were obtained by Di Benedetto-Trudinger \cite{DT}, Baroni-Colombo-Mingione \cite{BCM} and the references therein. See also Kinnunen-Lehrb\"{a}ck-V\"{a}h\"{a}kangas \cite{KLM}, Marcellini \cite{Ma, Ma0} and Mingione-R\v{a}dulescu \cite{MinR} for other results and literature reviews.

Quasiminima provide a natural setting for studying regularity theory of differential equations from a purely variational point of view, as demonstrated by Di Benedetto-Trudinger \cite{DT}, Giaquinta-Giusti \cite{GG1,GG2} and their references. Regularity methods established for quasiminima have the advantage to rely only on energy estimates. Furthermore, they offer a unifying aspect, indeed a large class of nonlinear differential equations with similar growth conditions are embraced by the same class of quasiminima. The main advantage of the notion of quasiminimizer of \eqref{J} is that it simultaneously covers a large class of problems where the variational integrand $F:\Omega\times\mathbb R\times \mathbb R\to\mathbb R$ satisfies the Carath\'eody conditions and
\[
\Theta H(x,z)\le F(x,u,z)\le\Lambda H(x,z),
\quad 0<\Theta<\Lambda<\infty,
\]
for every $x \in \Omega$ and $v,z\in \mathbb{R}$.

\noindent We point out that we develop our study of regularity for local quasiminima $u$ of \eqref{J} without assuming any other extra assumptions on $u$ except for the natural set up $N^{1,1}_{\loc}({\Omega})$ with $H(\cdot,g_u)\in L^1_{\loc}(\Omega)$. Here, $N^{1,1}_{\loc}({\Omega})$ denotes the Newtonian–Sobolev space on a metric measure space, see \cite{BB, FHK, H, HKST, KKM, Sh1}.

The regularity theory for minima and quasiminima in the general context of metric measure spaces poses additional difficulties. 
In metric measure spaces, it is not clear how to consider the Euler-Lagrange equation, but it is possible to apply variational methods. Certainly, metric measure spaces do not necessarily possess a smooth structure, and as a result, directions and partial derivatives may not be available. The initial observation to make is that in the application of variational methods, it is the modulus of the gradient, rather than the gradient itself, that plays a crucial role. In fact, Sobolev spaces on a metric measure space can be defined in terms of the modulus of the gradient without the notion of distributional derivatives. Recently, the study of Sobolev spaces without a differentiable structure and a variational theory of the $p$-energy functionals in metric measure spaces have attracted many researchers. Without claiming to provide an exhaustive literature, we refer to Bj\"{o}rn \cite{B} and Bj\"{o}rn-MacManus-Shanmugalingam \cite{BMS} for boundary regularity results for quasiminima. Also Kinnunen-Shanmugalingam \cite{KS} proved some local properties of quasiminima of the $p$-energy integral. Bj\"{o}rn-Bj\"{o}rn-Shanmugalingam \cite{BBS} studied the Dirichlet problem for $p$-harmonic functions.  Moreover, Cheeger \cite{C} focused on differentiability results and Kinnunen-Martio \cite{KM} on potential theory. Kinnunen-Marola-Martio \cite{KMM} proved Harnack principle and Kinnunen-Shanmugalingam \cite{KS} showed some local properties of quasiminima of the $p$-energy integral always in the context of metric spaces. Recently, higher integrability results for the gradient of quasiminima for double phase problems were obtained by Kinnunen-Nastasi-Pacchiano Camacho \cite{KNP}.

In the present work, by taking into account the structure of the underlying metric measure space, we use intrinsic type arguments to obtain regularity properties of quasiminima. In this direction, we develop an approach which presents the advantage to work with the usual estimates as starting point and it enables the study of the double phase functional $H(x,g_{u})$ as a sort of replacement of the generalization of the modulus of the gradient in the general context of metric measure spaces to the power $p$, that is $g_{u}^p$. 
The methods of studying regularity properties of double phase problems are strictly related to frozen functionals. Basically, the presence of the term related to the coefficient function $a(\cdot)$ is reduced to be a perturbation of the $p$-phase problem. 
This permits to use some classical results related to perturbation such as Schauder estimates \cite{Manfredi} and Calderón-Zygmund theory \cite{DM}. For example in the euclidean case, frozen functionals fall in the class of functionals with general growth conditions considered by Lieberman \cite{Lieberman}, for which regularity results are well known. To our knowledge, this is the first study where an approach based on frozen functionals is extended to the general context of metric measure spaces. 
We note that, with the use of this method, the constants in the estimates obtained do not depend on the  coefficient $a_0$ and this is a crucial difference from the corresponding results in \cite{NP}. 

The main challenge of the present manuscript is to provide self-contained and transparent proofs of new and more general theorems, which lay the basis for further studies and investigations in the theory of regularity for problems with non standard growth conditions. The motivation behind this study comes from the paper by Baroni-Colombo-Mingione \cite{BCM}. More specifically, the first step of the study consists in proving the local boundedness of the local quasiminimizer (Section \ref{Boundedness}). Then taking into account the nature of ellipticity of the functional, that implies a switch between the $p$ and $q$ rates accordingly to the value of $a(\cdot)$, we conduct a careful analysis on phases. This led to new findings regarding the so called frozen functionals in the general context of metric measure spaces (Section \ref{frozen}). This permits to establish the H\"older continuity of $u$, without assuming a-priori its local boundedness (Section \ref{Sec4}). The last section (Section \ref{Sec5}) is dedicated to a Harnack inequality result. These results are mainly obtained by assuming only condition \eqref{pqcond}, which consists in a bound on the ratio $q/p$, in dependence of the constant $Q$ which plays the role of dimension related to the measure $\mu$ and under certain regularity conditions (see \eqref{aalpha}) on the coefficient $a(\cdot)$.

\section{Notation and preliminaries}\label{Sec2}
Let $(X, d, \mu)$ be a complete metric measure space, where $d$ is the metric and $\mu$ is a doubling Borel regular measure, that is, there exists a constant $C_D \geq 1$ such that 
\begin{equation}\label{doubling}
0<\mu(B_{2r})\leq C_D \mu(B_r)<\infty,
\end{equation} 
for every open ball $B_r=B_r(x)=\{x\in X:d(y,x)<r\}$ with center $x\in X$ and radius $0<r<\infty$. 

The next lemma provides a notion of dimension of the space related to a doubling measure $\mu$.

\begin{lemma}[\cite{BB}, Lemma 3.3]\label{lemm3.3}
	Let $(X, d, \mu)$ be a metric measure space with a doubling measure $\mu$. Then, for every $0<r\le R<\infty$, $x \in X$ and $y \in B_R(x)$, 
	\begin{equation}\label{s}
		\frac{\mu(B_{r}(y))}{\mu(B_R(x))}\geq C\left(\frac{r}{R}\right)^Q, 
	\end{equation}
	where $Q=\log_2C_D$ and $C=C_D^{-2}$.
\end{lemma}

We observe that a complete metric metric measure space with a doubling measure is proper, which means that it is closed and bounded subsets are compact (see \cite[Proposition 3.1]{BB}).
Now we introduce the definition of upper gradient which permits to generalize the modulus of the gradient in the Euclidean case to the metric setting. 
More details can be found in the book \cite{BB} by  Bj\"{o}rn and Bj\"{o}rn.

\begin{definition}\label{pweakuppgrad}
	Let $g$ be a nonnegative Borel function. $g$ is an upper gradient of $u: X \to [-\infty,\infty]$ if, for all rectifiable paths $\gamma$ connecting the endpoints $x$ and $y$, we have 
	\begin{equation}\label{upgrad}
		|u(x)-u(y)|\leq \int_{\gamma}g\, \dd s, 
	\end{equation}
	whenever $u(x)$ and $u(y)$ are both finite and $\int_{\gamma}g \, \dd s= \infty$ otherwise. 
	Moreover, if a nonnegative measurable function $g$ satisfies \eqref{upgrad} for $p$-almost every path (that is, with the exception of a path family of zero $p$-modulus, see Definition 1.33 in \cite{BB}), then $g$ is called a $p$-weak upper gradient of $u$.
\end{definition}

Let $1\leq p<\infty$ and $\Omega$ be an open set of $X$. We define
\[
\Vert u\Vert_{N^{1,p}(\Omega)}=\Vert u\Vert_{L^{p}(\Omega)}+\inf\Vert g\Vert_{L^{p}(\Omega)},
\]
where the infimum is taken over all upper gradients $g$ of $u$.
We consider the set of functions $u\in L^p(\Omega)$ having upper gradient $g\in L^p(\Omega)$ and let
\begin{equation*}
	\widetilde{N}^{1,p}(\Omega)
	=\lbrace u:\Vert u\Vert_{N^{1,p}(\Omega)}<\infty\rbrace.
\end{equation*}
The Newtonian space is defined by
\[
N^{1,p}(\Omega)=\lbrace u:\Vert u\Vert_{N^{1,p}(\Omega)}<\infty\rbrace/\sim,
\]
where $u\sim v$ if and only if $\Vert u-v\Vert_{N^{1,p}(\Omega)}=0$. More details can be found in \cite{Sh1}.\\
A function $u$ is said to belong to the local Newtonian space $N^{1,p}_{\loc}(\Omega)$ if
$u\in N^{1,p}(\Omega')$ for all $\Omega'\Subset \Omega$, that is $\Omega'$ is compactly contained in $\Omega$ (see \cite[Proposition 2.29]{BB}).
It is known (see  \cite[Theorem 2.5]{BB}) that if $u$ has an upper gradient $g\in L^p(\Omega)$, then there is a unique minimal $p$-weak upper gradient $g_u\in L^p(\Omega)$, with
$g_u\le g$ $\mu$-almost everywhere for all $p$-weak upper gradients $g\in L^p(\Omega)$ of $u$.
Moreover, the minimal $p$-weak upper gradient is unique up to sets of measure zero.
Let $u\in N^{1,p}(\Omega)$ we have
\[
\Vert u\Vert_{N^{1,p}(\Omega)}=\Vert u\Vert_{L^{p}(\Omega)}+\Vert g_u\Vert_{L^{p}(\Omega)},
\]
where $g_u$ is the minimal $p$-weak upper gradient of $u$.\\
We define $N^{1,q}_0(\Omega)$ to be the set of functions $u\in N^{1,q}(X)$ that are zero on $X\setminus\Omega$ $\mu$-a.e. The space $N_0^{1,q}(\Omega)$ is equipped with the norm $\Vert\cdot\Vert_{N^{1,q}}$. We remark that if $\mu(X\setminus\Omega) = 0$, then $N^{1,q}_0(\Omega)=N^{1,q}(X)$. As a consequence, we must assume that $\mu(X \setminus \Omega) > 0$.\\
We denote by 
\[
u_{B}=\dashint_{B} u\, \dd\mu
=\frac{1}{\mu(B)}\int_{B}u\,\dd\mu.
\] the integral average of $u$ over a ball $B$.

We suppose that $X$ supports the following Poincar\'{e} inequality.

\begin{definition}
Let $1\le p<\infty$. 
A metric measure space $(X,d,\mu)$ supports a weak $(1, p)$-Poincar\'{e} inequality if there exist a constant $C_{PI}$ and a dilation factor $\lambda \geq 1 $ such that 
\[
	\dashint_{B_r} |u-u_{B_r}|\, \dd\mu
	\leq C_{PI} r \left(\dashint_{B_{\lambda r}}g_u^p \, \dd\mu\right)^{\frac{1}{p}},
\]
for every ball $B_r$ in $X$ and for every $u\in L^1_{\loc}(X)$.
\end{definition}

We note that the Poincar\'{e} inequality satisfies the following self-improving property, as shown in \cite[ Theorem 1.0.1]{KZ} by Keith and Zhong, see also \cite[Theorem 4.30]{BB}.
\begin{theorem}\label{kz}
	Let $(X, d, \mu)$ be a complete metric measure space with a doubling measure $\mu$ and a weak $(1,p)$-Poincar\'{e} inequality with $p>1$.
	Then there exists $\varepsilon>0$ such that $X$ supports a weak $(1, q)$-Poincar\'{e} inequality for every $q>p-\varepsilon$. 
	Here, $\varepsilon$ and the constants associated with the $(1, q)$-Poincar\'e inequality depend only on $C_D$, $C_{PI}$ and $p$.
\end{theorem}

The next theorem states that the Poincar\'e inequality implies a Sobolev--Poincar\'e inequality (see \cite[Theorem 4.21 and Corollary 4.26]{BB}).

\begin{theorem}\label{sstars}
	Let $\mu$ be a doubling measure and let $X$ support a weak $(1,p)$-Poincar\'{e} inequality. We consider $Q=\log_2C_D$ as in \eqref{s},  
	$1\le p^*\le\frac{Qp}{Q-p}$ for $1\le p<Q$ and $1\leq p^*<\infty$ for $Q\leq p<\infty$.
	Then $X$ supports a weak $(p^*,p)$-Poincar\'{e} inequality,
	that is, there exists a constant $C=C(C_D,C_{PI},p)$ such that
	\[
		\left(\dashint_{B_r} |u-u_{B_r}|^{p^*} \,\dd\mu\right)^{\frac{1}{p^*}}
		\leq Cr\left(\dashint_{B_{2\lambda r}}g_u^p \,\dd\mu\right)^{\frac{1}{p}},
	\]
	for every ball $B_r$ in $X$ and every $u \in L^1_{\loc}(X)$. 
\end{theorem}

Now we introduce the hypotheses set on the nonnegative coefficient function $a:X\to[0,\infty)$ in \eqref{J}. We suppose that $a$ is $\alpha$-H\"older continuous with respect to a quasi-distance related to the the underlying measure $\mu$. That is, there exists $0<\alpha\le1$, such that
\begin{equation}\label{aalpha}
	[a]_{\alpha}= \sup_{x,y \in \Omega, x\neq y} \dfrac{|a(x)-a(y)|}{\delta_{\mu}(x,y)^{\alpha}}<\infty,
\end{equation}
where $\delta_{\mu}$ is a quasi-distance defined as
\[
\delta_{\mu}(x,y)=\bigl(\mu(B_{d(x,y)}(x))+\mu(B_{d(x,y)}(y))\bigr)^{1/Q},\quad x,y \in X,\, x\ne y,
\]
with $Q=\log_2C_D$ as in \eqref{s} and $\delta_{\mu}(x,x)=0$. The following remark states that if $\mu$ is $Q$-Ahlfors--David regular, \eqref{aalpha} is equivalent to the standard H\"older continuity with exponent $\alpha$.
\begin{remark}
A measure is called Ahlfors--David regular, if there exist constants $0<C_1\le C_2<\infty$ such that
\begin{equation}\label{ahlfors}
C_1r^Q\le\mu(B_r(x))\le C_2r^Q,
\end{equation}
for every $x\in X$ and $0<r\le\diam(X)$.
If the measure $\mu$ is Ahlfors--David regular, then $\delta_{\mu}(x,y)\approx d(x,y)$ for every $x,y$ and, consequently,
$[a]_\alpha<\infty$ if and only if $a(\cdot)$ is H\"older continuous with the exponent $\alpha$.
\end{remark}
\textit{Throughout this paper, we assume 
\begin{equation*}
1\leq\frac{q}{p}\le1+\frac{\alpha}{Q}, 
\quad p>1,
\end{equation*}
where $\alpha \in (0,1]$ is as in \eqref{aalpha} and  $Q=\log_2C_D$ is as in \eqref{s}. We suppose that $(X, d, \mu)$ is a complete, metric measure space with metric $d$ and a doubling Borel regular measure $\mu$. We consider $\Omega\subset X$ an open subset and denote with $\Omega'$ any open and bounded subset of $\Omega$ such that $\Omega'\Subset \Omega$. Moreover, we assume that $X$ supports a weak $(1,p)$-Poincar\'{e} inequality. We use $C$ for positive constants, keeping also track of dependencies in parentheses. In particular, 
$C(\data)=C(C_D,C_{PI},\lambda,p,q,\alpha,[a]_\alpha).$}

\section{Energy estimates}\label{DeGiorgi}

We give the following definition of local quasiminimizer.

\begin{definition}\label{lqm}
	A function $u\in N^{1,1}_{\loc}({\Omega})$ with $H(\cdot,g_u)\in L^1_{\loc}(\Omega)$ is a local quasiminimizer on $\Omega$, if there exists a constant $K\geq 1$ such that 
	\[
		\int_{\Omega'\cap\{u\ne v\}}H(x,g_u)\,\dd \mu
		\leq K\int_{\Omega'\cap\{u\ne v\}}H(x,g_v)\,\dd \mu,
\]
	for every open subset $\Omega'\Subset\Omega$ and for every function $v\in N^{1,1}(\Omega')$ with $u-v\in  N^{1,1}_0(\Omega')$.
\end{definition}

In the rest of this section, we provide a double phase Caccioppoli inequality, that is a local energy estimate for the double phase problem.

\begin{lemma}\label{DeGiorgiLemma}
Assume that  $u\in N^{1,1}_{\loc}({\Omega})$ with $H(\cdot,g_u)\in L^1_{\loc}(\Omega)$  is a local quasiminimizer in $\Omega$. Then there exists $C= C(K,q)>0$ such that for any choice of concentric balls $B_\rho\subset B_R\Subset\Omega$ and all $k\in\mathbb{R}$, the following inequality 
\begin{align}\label{5.4}\int_{B_{\rho}}H(x,g_{(u-k)_{\pm}})\, \dd \mu	& \leq C\int_{B_R} H\left(x,\frac{(u-k)_{\pm}}{R-\rho}\right)\, \dd\mu
\end{align}
is satisfied, where $(u-k)_+= \max\{u-k,0\}$ and $(u-k)_-= \max\{-(u-k),0\}$ . 
\end{lemma}

\begin{remark}\label{caccipoliremark} 
We note that inequality \eqref{5.4} with $(u-k)_+$ is equivalent to the following
\begin{align}\label{othercacciopoli}\int_{S_{k,\rho}}H(x,g_u)\, \dd \mu	& \leq C\int_{S_{k,R}} H\left(x,\frac{(u-k)_+}{R-\rho}\right)\, \dd\mu,
\end{align}
where $(u-k)_+= \max\{u-k,0\}$ and $S_{k, r}= \{x \in B_r \cap {\Omega}: u(x)>k\}$.
Analogously, inequality \eqref{5.4} with $(u-k)_-$ is equivalent to the following
\begin{align}\label{cacciopoliminus}\int_{s_{k,\rho}}H(x,g_u)\, \dd \mu	& \leq C\int_{s_{k,R}} H\left(x,\frac{(u-k)_+}{R-\rho}\right)\, \dd\mu,
\end{align}
where $(u-k)_-= \max\{-(u-k),0\}$ and $s_{k, r}= \{x \in B_r \cap {\Omega}: u(x)<k\}$. 

\end{remark}

\begin{proof}[Proof of Lemma \ref{DeGiorgiLemma}]
By Remark \ref{caccipoliremark}, it is enough to prove \eqref{othercacciopoli} and \eqref{cacciopoliminus}. We work with the case $(u-k)_+$, since  the other is treated similarly.

Let $\tau$ be a $\dfrac{1}{R-\rho}$-Lipschitz cutoff function so that $0\leq \tau \leq1$, $\tau=1$ on $B_\rho$ and the support of $\tau$ is contained in $B_R$.
We consider 
\begin{equation*}
w=u- \tau (u-k)_+= \begin{cases}
(1-\tau)(u-k)+k\quad \mbox{in } S_{k,R},\\
u \quad \quad \quad\quad\quad\quad\quad\quad\mbox{ otherwise.}\\
\end{cases}
\end{equation*}
Using Leibniz' rule,
\begin{equation*}
g_{w}\leq (u-k)g_{\tau}+(1-\tau)g_u\quad \quad \quad\quad\textrm{in }S_{k,R}.
\end{equation*}
We observe that
\begin{equation*}
g_{w} ^l\leq 
2^{l-1}\left(g_{u} ^l(1-\chi_{S_{k, \rho}})+\left(\dfrac{u-k}{R-\rho}\right)^l\right) \quad \mbox{ in $S_{k, R}$, where $l=p,q$ }.
\end{equation*}	
Since $u$ is a local quasiminimizer and $u-v\in N_{0}^{1,1}(S_{k, R})$, by Definition \ref{lqm} we obtain
\begin{align}\label{5.7}
\int_{S_{k, \rho}}H(x,g_u)\, \dd \mu  &\leq\int_{S_{k, R}}H(x,g_u)\, \dd\mu \leq K\int_{S_{k, R}}H(x,g_w)\, \dd\mu \nonumber \\
& \leq K\left(2^q\int_{S_{k, R}}H\left(x,\frac{u-k}{R-\rho}\right)\, \dd\mu+2^{q}\int_{S_{k, R}\setminus S_{k,\rho}}H(x,g_u)\, \dd \mu\right).
\end{align}	
\noindent By adding $K 2^{q}\int_{S_{k, \rho}}H(x,g_u)\, \dd \mu$ to both sides of (\ref{5.7}), we get 
\begin{align}\label{5.7bis}
(1&+K 2^{q}) \int_{S_{k, \rho}}H(x,g_u)\, \dd \mu   \leq K2^{q}\left(\int_{S_{k, R}}H\left(x,\frac{u-k}{R-\rho}\right)\, \dd\mu+ \int_{S_{k, R}}H(x,g_u)\, \dd \mu\right).
\end{align}	
Let $\theta= \dfrac{K 2^{q}}{1+K 2^{q}}<1$, then 
\begin{align*}
\int_{S_{k, \rho}}H(x,g_u)\, \dd \mu   \leq \theta\left(\int_{S_{k, R}}H\left(x,\frac{u-k}{R-\rho}\right)\, \dd\mu+ \int_{S_{k, R}}H(x,g_u)\, \dd \mu\right).
\end{align*}

At this point we can use Lemma 6.1 of \cite{G}, with $\alpha=q$, $\beta=p$, $$Z(\rho)=\int_{S_{k, \rho}}H(x,g_u)\, \dd \mu, \quad A=  \theta\int_{S_{k, R}}(u-k)^p\, \dd \mu, \quad B= \theta\int_{S_{k, R}}a(x)(u-k)^q\, \dd \mu,$$to get
\begin{align*}
	\int_{S_{k, \rho}}H(x,g_u)\, \dd \mu &	\leq C(q,K)\left(\theta\int_{S_{k, R}}H\left(x,\frac{u-k}{R-\rho}\right)\, \dd\mu\right).
\end{align*}
That is
\begin{align*}
	\int_{S_{k, \rho}}H(x,g_u)\, \dd \mu &	\leq C(q,K)\int_{S_{k, R}}H\left(x,\frac{u-k}{R-\rho}\right)\, \dd\mu.
\end{align*}
\end{proof}

\begin{remark}\label{Proposition 3.3KS}
	Let $\Omega$ be an open subset of $X$. Notice that if $u$ is a local quasiminimizer then $-u$ is also a local quasiminimizer. Therefore, by Lemma \ref{DeGiorgiLemma},  we get that $-u$ satisfies \eqref{5.4}.
\end{remark}

\begin{definition}[Double phase De Giorgi Classes]\label{DGclass} We write $DG_{H}^{\pm}(\Omega)$ for the collection of functions  $u\in N^{1,1}_{\loc}({\Omega})$ with $H(\cdot, g_u)\in L^1_{\loc}(\Omega)$ satisfying \eqref{5.4} for any choice of concentric balls $B_\rho\subset B_R\Subset\Omega$ and all $k\in\mathbb{R}$. The double phase De Giorgi class is then defined as
$$
DG_{H}(\Omega)=DG_{H}^{+}(\Omega)\cap DG_{H}^{-}(\Omega).
$$
\end{definition}

\section{Boundedness of local quasiminima}\label{Boundedness}

The aim of this section is to prove that local quasiminima, in the sense of the Definition \ref{lqm}, are locally bounded. Before starting the proof, we report the double phase Sobolev-Poincar\'e inequality for functions which vanish on a large set, which has been proven in \cite{KNP}.

\begin{lemma}[\cite{KNP}, Lemma 3.4] \label{lemma3.4JAC}Assume that $u\in N^{1,1}_{\textrm{loc}}(\Omega)$ with $H(\cdot, g_u)\in L^1_{\textrm{loc}}(\Omega)$. Let $B_r$ be a ball and assume that there exist $\gamma$, $0<\gamma<1$, such that 
\begin{equation*}
    \mu(\lbrace x\in B_r: \vert u(x)\vert>0\rbrace)\leq \gamma\mu(B_r).
\end{equation*}
Then, there exists a constant $C= C(\data, \gamma)$ and exponents $0<d_2<1\leq d_1<\infty$, with $d_1= d_1(\data)$ and $d_2= d_1(\data)$, such that

$$
\left(\dashint_{B_r} H\left(x,\frac{u}{r}\right)^{d_1}\dd \mu\right)^{1/d_1}\leq C\left(1+\Vert g_u\Vert_{L^{p}(B_{2\lambda r})}^{q-p}\mu(B_{2\lambda r})^{\frac{\alpha}{Q}-\frac{q-p}{p}}\right)\left(\dashint_{B_{2\lambda r}}H(x,g_u)^{d_2}\dd \mu\right)^{1/d_2},
$$
whenever $B_{2\lambda r}\Subset\Omega$. 
\end{lemma}

\begin{remark}\label{remarkDPPI}
Let $B_r$ be a ball, with $0<r<\frac{\diam(X)}{3}$ and centre $x_0$. Let $u\in N_0^{1,1}(B_r)$ with $H(\cdot, g_u)\in L^{1}_{\textrm{loc}}(B_r)$. Notice that, since $r<\frac{\diam(X)}{3}$, the set $X\setminus B_{\frac{3r}{2}}$ is non-empty. Furthermore, the Poincar\'e inequality implies that $\partial(B_{\frac{3r}{2}})$ is also non-empty, see \cite{BB}, i.e. there exists a point $x'$ on the sphere $\lbrace x\in X: d(x,x_0)=\frac{3r}{2}\rbrace$. Then, the set $B_{2r}\setminus B_r$ contains the ball $B'=B(x',\frac{r}{2})$. On the other hand, $B_{2r}\subset 7B'=B(x',\frac{7r}{2})$, Lemma \ref{lemm3.3} and the doubling property of $\mu$ implies
\begin{equation}\label{DPPI1}
\frac{\mu(B_r)}{\mu(B_{2r})}\leq 1-\frac{\mu(B')}{\mu(B_{2r})}\leq 1-C_D^{-2}\left(\frac{1}{4}\right)^{Q}=\gamma=\gamma(Q,C_D)<1.   
\end{equation}
Now, if we define $A=\lbrace x\in B_{2r}: \vert u(x)\vert>0\rbrace\subset B_r$. Since $u\in N^{1,1}_0(B_r)$, by \eqref{DPPI1}, we have $\mu(A)\leq\mu(B_r)\leq \gamma\mu(B_{2r})$. Therefore, by Lemma \ref{lemma3.4JAC} there exists $C= C(\data)$ and exponents $0<d_2<1\leq d_1<\infty$, with $d_1= d_1(\data)$ and $d_2= d_2(\data)$, such that
$$
\left(\dashint_{B_r} H\left(x,\frac{u}{r}\right)^{d_1}\dd \mu\right)^{1/d_1}\leq C\left(1+\Vert g_u\Vert_{L^{p}(B_{2\lambda r})}^{q-p}\mu(B_{2\lambda r})^{\frac{\alpha}{Q}-\frac{q-p}{p}}\right)\left(\dashint_{B_{2\lambda r}}H(x,g_u)^{d_2}\dd \mu\right)^{1/d_2}.
$$
Since $u\in N^{1,1}_0(B_r)$, then $g_u=0$ $\mu$-a.e. on $X\setminus B_r$. Therefore, by the doubling property of the measure, the right-hand side of the previous inequality becomes
\begin{align*}
\left(\dashint_{B_r} H\left(x,\frac{u}{r}\right)^{d_1}\dd \mu\right)^{1/d_1}&\leq C\left(1+\Vert g_u\Vert_{L^{p}(B_{r})}^{q-p}\mu(B_{2\lambda r})^{\frac{\alpha}{Q}-\frac{q-p}{p}}\right)\left(\dashint_{B_{r}}H(x,g_u)^{d_2}\dd \mu\right)^{1/d_2}\\
&\leq C\left(1+\Vert g_u\Vert_{L^{p}(B_{r})}^{q-p}\mu(B_{r})^{\frac{\alpha}{Q}-\frac{q-p}{p}}\right)\left(\dashint_{B_{r}}H(x,g_u)^{d_2}\dd \mu\right)^{1/d_2}.
\end{align*}
\end{remark}
As a consequence of the considerations made in the previous remark, we have the following double phase Sobolev-Poincar\'e inequality for functions with zero boundary values.
\begin{lemma}\label{DPPIzero} Let  $u\in N^{1,1}_{\textrm{loc}}(\Omega)$ be such that $H(\cdot, g_u)\in L^1_{\textrm{loc}}(\Omega)$. Let $B_r$ be a ball, with $0<r<\frac{\diam(X)}{3}$. Assume $u\in N^{1,1}_{0}(B_r)$.
Then, there exists a constant $C= C(\data)$ and exponents $0<d_2<1\leq d_1<\infty$, with $d_1= d_1(\data)$ and $d_2= d_1(\data)$, such that
$$
\left(\dashint_{B_r} H\left(x,\frac{u}{r}\right)^{d_1}\dd \mu\right)^{1/d_1}\leq C\left(1+\Vert g_u\Vert_{L^{p}(B_{r})}^{q-p}\mu(B_{r})^{\frac{\alpha}{Q}-\frac{q-p}{p}}\right)\left(\dashint_{B_{r}}H(x,g_u)^{d_2}\dd \mu\right)^{1/d_2}.
$$
\end{lemma}
The following two lemmata show some energy estimates.  

\begin{lemma}\label{lemmaestimate1}  Let  $u\in N^{1,1}_{\textrm{loc}}(\Omega)$ be such that $H(\cdot, g_u)\in L^1_{\textrm{loc}}(\Omega)$. Assume $u \in DG_H(\Omega)$. Let $0<\frac{R}{2}<\rho<s\leq R\leq \min\lbrace 1, \frac{\diam(X)}{6}\rbrace$ and let $B_{\rho}\subset B_s\subseteq B_R\Subset \Omega$ be concentric balls. Then, there exist a constant $C= C(\data)$ and an exponent $d_2= d_2(data)< 1$ such that
    \begin{align*}
\int_{B_{\rho}}&H\left(x,(u-k)_+\right)  \dd\mu \leq \frac{C}{(s-\rho)^{q-p}}\left(1+\Vert u\Vert_{N^{1,p}(B_R)}^{q-p}\mu(B_R)^{\frac{\alpha}{Q}-\frac{q-p}{p}}\right)\left(\dfrac{\mu(S_{k, s})}{\mu(B_{s})}\right)^{\frac{1}{d_2}-1}\int_{B_s}H\left(x,\frac{(u-k)_+}{s-\rho}\right) \dd \mu.
\end{align*}
\end{lemma}

\begin{proof}
 Let $0<\frac{R}{2}<\rho<s\leq R\leq \min\lbrace 1, \frac{\diam(X)}{6}\rbrace$. To simplify notation, we define $t=\frac{s+\rho}{2}$. Notice that, by definition, $\rho<t<s$. Since $u\in DG_{H}(\Omega)$, in particular $u\in DG_{H}^+(\Omega)$. Therefore, by the double phase Caccioppoli inequality \eqref{5.4}, we have
\begin{align}\label{ks2}
	\int_{B_t} H(x, g_{(u-k)_+}) \dd \mu&\leq  C \int_{B_s} H\left(x,\frac{(u-k)_+}{s-t}\right)\, \dd\mu \nonumber\\
& \leq  C 2^{q}  \int_{B_s} H\left(x,\frac{(u-k)_+}{s-\rho}\right)\, \dd\mu \nonumber\\
&= C \int_{B_s} H\left(x,\frac{(u-k)_+}{s-\rho}\right)\, \dd\mu,
\end{align}
where $C$ depends on $K$ and $q$. Let $\tau$ be a $\dfrac{1}{s-\rho}$-Lipschitz cutoff function so that $0\leq \tau \leq1$, $\tau=1$ on $B_\rho$ and the support of $\tau$ is contained in $B_t$. Let $w=\tau (u-k)_+\in N_0^{1,1}(B_t)$. By Leibniz rule, we have $$g_w \leq g_{(u-k)_+}\tau +(u-k)_+ g_{\tau}\leq g_{(u-k)_+}+ \frac{1}{s-\rho}(u-k)_+,\qquad\textrm{on }B_t.$$
Using inequality \eqref{ks2}, we get
\begin{align}\label{15notes}
	\int_{B_t} H(x,g_w) \dd \mu&\leq  2^{q-1}\int_{B_t} H(x,g_{(u-k)_+}) \dd \mu+2^{q-1}\int_{B_{t}} H\left(x,\frac{(u-k)_+}{s-\rho}\right)\dd \mu \nonumber \\  
	&\leq C \, 2^{q-1} \int_{B_s} H\left(x,\frac{(u-k)_+}{s-\rho}\right)\dd \mu 
+2^{q-1} \int_{B_s}H\left(x,\frac{(u-k)_+}{s-\rho}\right) \dd \mu \nonumber \\  
	&=  C \int_{B_s}H\left(x,\frac{(u-k)_+}{s-\rho}\right) \dd \mu.
\end{align}

\noindent By H\"older inequality, the doubling property, the definition of $w$ and Lemma \ref{DPPIzero}, there are a constant $C=C({\rm data})$ and exponents $0<d_2<1\leq d_1<\infty$, such that

\begin{align}\label{ks40}
\dashint_{B_{\rho}}H\left(x,(u-k)_+\right)  \dd\mu&\leq \dashint_{B_{\rho}}H\left(x,\frac{(u-k)_+}{t}\right) \dd\mu\nonumber\\
&\leq   
\left(\dashint_{B_{\rho}}H\left(x,\frac{(u-k)_+}{t}\right) ^{d_1} \dd\mu\right)^{\frac{1}{d_1}}\nonumber\\ 
&=\left(\dashint_{B_{\rho}}H\left(x,\frac{\tau(u-k)_+}{t}\right) ^{d_1} \dd\mu\right)^{\frac{1}{d_1}}\nonumber\\ 
 &= \left(\dashint_{B_{\rho}}H\left(x,\frac{w}{t}\right) ^{d_1} \dd\mu\right)^{\frac{1}{d_1}}\nonumber\\ 
 &\leq \left(\dfrac{\mu\big(B_{t}\big)}{\mu(B_{\rho})}\right)^{\frac{1}{d_1}}\left(\dashint_{B_{t}} H\left(x,\frac{w}{t}\right) ^{d_1} \dd\mu\right)^{\frac{1}{d_1}} \nonumber\\	
 &\leq C\left(1+\|g_w\|^{q-p}_{L^p(B_{t})} \mu(B_{t})^{\frac{\alpha}{Q}-\frac{q-p}{p}}\right)\left(\dashint_{B_{t}}H\left(x,g_w\right) ^{d_2} \dd\mu\right)^{\frac{1}{d_2}}.
\end{align}
In the second to last inequality, we used the doubling property to estimate 
\begin{equation}\label{starextra}
    \left(\dfrac{\mu\big(B_{t}\big)}{\mu(B_{\rho})}\right)^{\frac{1}{d_1}}\leq C,
\end{equation}
where $C$ depends on the exponent $Q$ in \eqref{s}. By H\"{o}lder inequality, we have
\begin{align}\label{ks4first}
\left(\dashint_{B_{t}}H\left(x,g_w\right) ^{d_2} \dd\mu\right)^{\frac{1}{d_2}}&=\mu(B_{t})^{-\frac{1}{d_2}}\left(\int_{S_{k, t}}H\left(x,g_w\right)^{d_2} \dd\mu\right)^{\frac{1}{d_2}}\nonumber\\
&\leq \mu(B_{t})^{-\frac{1}{d_2}}\mu(S_{k, t})^{\frac{1}{d_2}-1}\int_{S_{k, t}}H\left(x,g_w\right) \dd\mu.
\end{align}
By \eqref{ks40}, \eqref{ks4first} and \eqref{15notes}, we get 
\begin{align}\label{ks4} 
&\int_{B_{\rho}}H\left(x,(u-k)_+\right)  \dd\mu\nonumber\\
&\leq C\left(1+\|g_w\|^{q-p}_{L^p(B_{t})} \mu(B_{t})^{\frac{\alpha}{Q}-\frac{q-p}{p}}\right)\left(\dfrac{\mu(S_{k, t})}{\mu(B_{t})}\right)^{\frac{1}{d_2}-1} \int_{S_{k, t}}H\left(x,g_w\right) \dd\mu \nonumber\\
 &\leq C\left(1+\|g_w\|^{q-p}_{L^p(B_t)} \mu(B_R)^{\frac{\alpha}{Q}-\frac{q-p}{p}}\right)\left(\dfrac{\mu(S_{k, t})}{\mu(B_{t})}\right)^{\frac{1}{d_2}-1}\int_{B_s}H\left(x,\frac{(u-k)_+}{s-\rho}\right) \dd \mu.
\end{align}
On the other hand, notice that 
\begin{align*}
\Vert g_w\Vert_{L^p(B_t)}^p&=\int_{B_t}g_w^p\dd\mu\nonumber\\
&\leq \int_{B_t}\left(g_{(u-k)_+}+\frac{(u-k)_+}{s-\rho}\right)^p\dd\mu\nonumber\\
&\leq 2^{p-1}\left(\int_{B_t}g_{(u-k)_+}^p\dd\mu+\int_{B_t}\left(\frac{(u-k)_+}{s-\rho}\right)^p\dd\mu\right)\\
&\leq C\left(\Vert g_u\Vert_{L^p(B_t)}^{p}+\frac{\Vert u\Vert^p_{L^p(B_t)}}{(s-\rho)^p}\right)\\
&\leq \frac{C}{(s-\rho)^p}\left(\Vert g_u\Vert_{L^p(B_t)}^{p}+\Vert u\Vert^p_{L^p(B_t)}\right).
\end{align*}
Since $u\in N^{1,1}_{\textrm{loc}}(\Omega)$, $H(\cdot,g_u)\in L^{1}_{\textrm{loc}}(\Omega)$, $B_R\Subset \Omega$ and the nature of the double phase functional, then $u\in N^{1,p}(B_R)$. Therefore,
\begin{align*}
\Vert g_w\Vert_{L^p(B_t)}^{q-p}&\leq\frac{C}{(s-\rho)^{q-p}}\left(\Vert g_u\Vert_{L^p(B_t)}^{p}+\Vert u\Vert^p_{L^p(B_t)}\right)^{\frac{q-p}{p}}\\
&\leq\frac{C}{(s-\rho)^{q-p}}\left(\Vert g_u\Vert_{L^p(B_t)}+\Vert u\Vert_{L^p(B_t)}\right)^{q-p}\\
&=\frac{C}{(s-\rho)^{q-p}}\Vert u\Vert_{N^{1,p}(B_t)}^{q-p}\\
&\leq\frac{C}{(s-\rho)^{q-p}}\Vert u\Vert_{N^{1,p}(B_R)}^{q-p}.
\end{align*}
By \eqref{ks4} and the last inequality, we then obtain
\begin{align}\label{new1}
\int_{B_{\rho}}&H\left(x,(u-k)_+\right)  \dd\mu\nonumber\\
&\leq \frac{C}{(s-\rho)^{q-p}}\left(1+\Vert u\Vert_{N^{1,p}(B_R)}^{q-p}\mu(B_R)^{\frac{\alpha}{Q}-\frac{q-p}{p}}\right)\left(\dfrac{\mu(S_{k, t})}{\mu(B_{t})}\right)^{\frac{1}{d_2}-1}\int_{B_s}H\left(x,\frac{(u-k)_+}{s-\rho}\right) \dd \mu.
\end{align}
\noindent Furthermore, we observe that $\mu(S_{k, t})\leq \mu(S_{k, s})$ and $\frac{1}{d_2}-1>0$. Thus, by \eqref{new1} we get
\begin{align*}
\int_{B_{\rho}}&H\left(x,(u-k)_+\right)  \dd\mu \leq \frac{C}{(s-\rho)^{q-p}}\left(1+\Vert u\Vert_{N^{1,p}(B_R)}^{q-p}\mu(B_R)^{\frac{\alpha}{Q}-\frac{q-p}{p}}\right)\left(\dfrac{\mu(S_{k, s})}{\mu(B_{t})}\right)^{\frac{1}{d_2}-1}\int_{B_s}H\left(x,\frac{(u-k)_+}{s-\rho}\right) \dd \mu.
\end{align*}
Lastly, by the doubling property of the measure (as we did in \eqref{starextra}), we can conclude
\begin{align*}
\int_{B_{\rho}}&H\left(x,(u-k)_+\right)  \dd\mu \leq \frac{C}{(s-\rho)^{q-p}}\left(1+\Vert u\Vert_{N^{1,p}(B_R)}^{q-p}\mu(B_R)^{\frac{\alpha}{Q}-\frac{q-p}{p}}\right)\left(\dfrac{\mu(S_{k, s})}{\mu(B_{s})}\right)^{\frac{1}{d_2}-1}\int_{B_s}H\left(x,\frac{(u-k)_+}{s-\rho}\right) \dd \mu,
\end{align*}
as wanted.
\end{proof}

\begin{lemma}\label{lemmaestimate2} Let $\Omega''$ be an open and bounded subset of $\Omega$ such that $\Omega''\Subset \Omega$. Assume $u\in N^{1,1}_{\textrm{loc}}(\Omega)$ with $H(\cdot, g_u)\in L^1_{\textrm{loc}}(\Omega)$ belongs to $DG_H(\Omega)$. Let $0<\frac{R}{2}<\rho<s\leq R\leq \min\lbrace 1, \frac{\diam(X)}{6}\rbrace$, and concentric balls $B_{\rho}\subset B_s\subseteq B_R\subset \Omega''$. Then, for any positive numbers $h<k$, there exists a constant $C= C(\data, \Omega'', \Vert u\Vert_{N^{1,p}(\Omega'')})$, and an exponent $0<\theta= \theta(data)$, such that
   \begin{equation*}
	\int_{S_{k,\rho}}H(x,u-k) \dd\mu\leq \frac{C}{(s-\rho)^{q-p}}\left(\dfrac{\mu(S_{k, s})}{\mu(B_s)}\right)^{\theta} \int_{S_{h,s}} H\left(x,\frac{u-h}{s-\rho}\right) \dd\mu.
\end{equation*}
\end{lemma}
\begin{proof}
Let $0<\frac{R}{2}<\rho<s\leq R\leq \min\lbrace 1, \frac{\diam(X)}{6}\rbrace$. By Lemma \ref{lemmaestimate1}, there exists a constant $C= C(\data)$, and an exponent $d_2= d_2(\data)< 1$, such that
    \begin{align}\label{eqbeforeomega}
	\int_{B_{\rho}}&H\left(x,(u-k)_+\right)  \dd\mu\nonumber\\
	& \leq \frac{C}{(s-\rho)^{q-p}}\left(1+\Vert u\Vert_{N^{1,p}(B_R)}^{q-p}\mu(B_R)^{\frac{\alpha}{Q}-\frac{q-p}{p}}\right)\left(\dfrac{\mu(S_{k, s})}{\mu(B_{s})}\right)^{\frac{1}{d_2}-1}\int_{B_s}H\left(x,\frac{(u-k)_+}{s-\rho}\right) \dd \mu.
\end{align}

Since, by hypothesis $\frac{q}{p}\leq 1+ \frac{\alpha}{Q}$, then $\frac{\alpha}{Q}-\frac{q-p}{p}\geq 0$. Thus, by the definition of $\Omega''$ and the assumption $B_R\subset \Omega''$, we obtain
\begin{align}\label{5notes}
\left(1+\Vert u\Vert_{N^{1,p}(B_R)}^{q-p}\mu(B_R)^{\frac{\alpha}{Q}-\frac{q-p}{p}}\right)\leq \left(1+\Vert u\Vert_{N^{1,p}(\Omega'')}^{q-p}\mu(\Omega')^{\frac{\alpha}{Q}-\frac{q-p}{p}}\right)\leq C,
\end{align} where $C=C(\Omega'', \Vert u\Vert_{N^{1,p}(\Omega'')})$.
Now, using \eqref{eqbeforeomega} and \eqref{5notes}, we get
\begin{align}\label{algo}
\int_{B_{\rho}}H\left(x,(u-k)_+\right)  \dd\mu \leq \frac{C}{(s-\rho)^{q-p}}\left(\dfrac{\mu(S_{k, s})}{\mu(B_{s})}\right)^{\frac{1}{d_2}-1}\int_{B_s}H\left(x,\frac{(u-k)_+}{s-\rho}\right) \dd \mu.
\end{align}
with $C= C(\data, \Omega'', \Vert u\Vert_{N^{1,p}(\Omega'')})$.

Let $h<k$, then $(u-k)_+\leq (u-h)_+$. Therefore, we have
\begin{align*}
\int_{S_{k,\rho}}H(x,u-k) \dd\mu &\leq  \frac{C}{(s-\rho)^{q-p}}\left(\dfrac{\mu(S_{k, s})}{\mu(B_{\rho})}\right)^{\frac{1}{d_2}-1} \int_{S_{h,s}} H\left(x,\frac{u-h}{s-\rho}\right) \dd \mu.
\end{align*}
That is
\begin{equation*}
	\int_{S_{k,\rho}}H(x,u-k) \dd\mu\leq \frac{C}{(s-\rho)^{q-p}}\left(\dfrac{\mu(S_{k, s})}{\mu(B_s)}\right)^{\theta} \int_{S_{h,s}} H\left(x,\frac{u-h}{s-\rho}\right) \dd\mu,
\end{equation*}
where $\theta=\frac{1}{d_2}-1>0$. 
    
\end{proof}

At this point, using Lemma \ref{lemmaestimate2}, we obtain the following theorem, which states that a function belonging to $DG_H(\Omega)$ is essentially locally bounded from above. 
A corresponding result for local quasiminima for double phase problems in the metric context was first proved in \cite{NP}, Theorem 4.5, assuming that the coefficient function $a$ is bounded away from zero. Since in this notes we work with weaker assumptions on the coefficient function, we focus on an intrinsic approach.

\begin{theorem}\label{Theorem 4.2}
	We consider $\Omega'' \Subset \Omega$ open and bounded. Assume $u\in N^{1,1}_{\textrm{loc}}(\Omega)$ with $H(\cdot, g_u)\in L^1_{\textrm{loc}}(\Omega)$ belongs to $DG_H(\Omega)$. Let $0<R\leq \min\lbrace 1, \frac{\diam(X)}{6}\rbrace$, such that $B_R \subset \Omega''$, and $k_0 \in \mathbb{R}$. Then, there exists $C=C(\data, R, \Omega'', \Vert u\Vert_{N^{1,p}(\Omega'')})>0$ such
	that $$\esssup_{B_{\frac{R}{2}}} u \leq k_0 + C \left(\dashint_{B_R} H(x,(u-k_0)_+) \, \dd\mu\right)^{\frac{1}{p}}.$$
	
\end{theorem}

\begin{proof}
For $n\in \mathbb{N}\cup \{0\}$, let $\rho_{n}=\frac{R}{2}\left(1+\frac{1}{2^{n}}\right)\leq R$ and $k_{n}=k_{0}+d\left(1-\frac{1}{2^n}\right)$, where $d>0$ will be chosen later. Then, $\rho_{0}=R$, $\rho_{n}\searrow \frac{R}{2}$ and $k_n\nearrow k_0+d$.
We apply Lemma \ref{lemmaestimate2} with $\rho=\rho_{i+1}$, $R=\rho_i$, $k=k_{i+1}$ and $h=k_i$ and we get
\begin{align}\label{8notes}
 \int_{S_{k_{i+1},\rho_{i+1}}}H(x,u-k_{i+1}) \dd\mu&\leq \frac{C}{(\rho_i-\rho_{i+1})^{q-p}}\left(\dfrac{\mu(S_{k_{i+1}, \rho_i})}{\mu(B_{R})}\right)^{\theta} \int_{S_{k_i,\rho_i}} H\left(x,\frac{u-k_i}{\rho_i-\rho_{i+1}}\right) \dd\mu\nonumber \\
 & = \frac{C}{(R 2^{-i}2^{-2})^{q-p}}\left(\dfrac{\mu(S_{k_{i+1}, \rho_i})}{\mu(B_{R})}\right)^{\theta} \int_{S_{k_i,\rho_i}} H\left(x,\frac{u-k_i}{R 2^{-i}2^{-2}}\right) \dd\mu\nonumber \\
& = \frac{C2^{i(2q-p)}}{R^{q-p}}\left(\dfrac{\mu(S_{k_{i+1}, \rho_i})}{\mu(B_{R})}\right)^{\theta} \int_{S_{k_i,\rho_i}} H\left(x,\frac{u-k_i}{R}\right) \dd\mu
\nonumber\\
& \leq \frac{C4^{iq}}{R^{2q-p}}\left(\dfrac{\mu(S_{k_{i+1}, \rho_i})}{\mu(B_R)}\right)^{\theta} \int_{S_{k_i,\rho_i}} H(x,u-k_i) \dd\mu
\nonumber\\
& \leq \frac{C4^{iq}}{R^{2q}}\left(\dfrac{\mu(S_{k_{i+1}, \rho_i})}{\mu(B_R)}\right)^{\theta} \int_{S_{k_i,\rho_i}} H(x,u-k_i) \dd\mu.
\end{align}
We observe that 
\begin{align*}
    d^{-p}(k_{i+1}-k_i)^p \mu(S_{k_{i+1}, \rho_i})&=
d^{-p}\int_{S_{k_{i+1}, \rho_i}}(k_{i+1}-k_i)^p \dd\mu\\
&\leq d^{-p}\int_{S_{k_{i+1}, \rho_i}}H(x, u-k_i) \dd\mu\\
&\leq d^{-p}\int_{S_{k_i, \rho_i}}H(x, u-k_i) \dd\mu.
\end{align*}
So, we have 
\begin{equation*}
    \psi_i= d^{-p}\int_{S_{k_i, \rho_i}}H(x, u-k_i) \dd\mu \geq d^{-p}(k_{i+1}-k_i)^p \mu(S_{k_{i+1}, \rho_i}). 
\end{equation*}
This implies
\begin{equation}\label{9notes}
 \mu(S_{k_{i+1}, \rho_i})\leq \psi_i d^{p}(k_{i+1}-k_i)^{-p}. 
\end{equation}
Therefore, by \eqref{8notes} and \eqref{9notes} we obtain 
\begin{align*}
   d^p \psi_{i+1}&\leq \frac{C4^{iq}}{R^{2q}} \left(\dfrac{\mu(S_{k_{i+1}, \rho_i})}{\mu(B_R)}\right)^{\theta} d^p\psi_i\nonumber\\ 
    &\leq \frac{C4^{iq}}{R^{2q}} \left(\psi_i d^{p}(k_{i+1}-k_i)^{-p}\right)^{\theta}d^p\psi_i \mu(B_R)^{-\theta}\nonumber\\
     &= \frac{C4^{iq}}{R^{2q}} (d2^{-1-i})^{-p\theta} d^{p\theta}\psi_i^{1+\theta} d^p \mu(B_R)^{-\theta}\nonumber\\
    &\leq  \frac{C4^{(1+\theta)qi}}{R^{2q}} \psi_i^{1+\theta}d^p \mu(B_R)^{-\theta}.
\end{align*}
That is,
\begin{equation}
 \psi_{i+1} \leq\frac{C4^{(1+\theta)qi}}{R^{2q}} \psi_i^{1+\theta}\mu(B_R)^{-\theta},
\end{equation}
for every $i\geq0$, where $C=C({\rm data}, \Omega'',\Vert u\Vert_{N^{1,p}(\Omega'')})$.
By using a standard iteration lemma (see Lemma 7.1, \cite{G})
we get \begin{equation*}
 \lim_{i\to \infty}\psi_i=\lim_{i\to \infty}d^{-p}\int_{S_{k_i, \rho_i}}H(x, u-k_i) \dd\mu=0,   
\end{equation*} provided that $d=CR^{\frac{-2q}{p\theta}} \left(\dashint_{B_R}H(x, (u-k_0)_+) \dd\mu\right)^{\frac{1}{p}}>0$. As a consequence, \begin{equation*}
    \int_{B_{\frac{R}{2}}} H(x, (u-(k_0+d))_+ \dd\mu=0
\end{equation*}
and so
\begin{equation}\label{quasiboundedness}
    u\leq k_0+d \quad\mbox{almost everywhere in $B_{\frac{R}{2}}$.}
\end{equation}
We conclude that
\begin{align*}
   \esssup_{B_{\frac{R}{2}}}u &\leq k_0 +d= k_0+ C \left(\dashint_{B_R}H\left(x,(u-k_0)_+\right) \, \dd\mu\right)^{\frac{1}{p}},
\end{align*}
where $C=C(\data, R, \Omega'', \Vert u\Vert_{N^{1,p}(\Omega'')})$.
\end{proof}
As a corollary of Theorem \ref{Theorem 4.2}, we obtain the main result of this section. 

\begin{corollary}[Local boundedness] Assume $u\in N^{1,1}_{\textrm{loc}}(\Omega)$ with $H(\cdot, g_u)\in L^1_{\textrm{loc}}(\Omega)$ is a local quasiminimizer. Then, 
$$
\Vert u\Vert_{L^{\infty}(\Omega')}<\infty
$$
holds, whenever $\Omega'\Subset\Omega$ is open and bounded. That is, $u$ is locally bounded.
\end{corollary}

\begin{proof}
Since $u$ is a local quasiminimizer, then by Lemma \ref{DeGiorgiLemma}, $u\in DG(\Omega)$. Therefore, $u$ satisfies the double phase Caccioppoli inequality \eqref{5.4}.
Let $\Omega'\Subset\Omega''\Subset\Omega$, and $0<R<\min\lbrace 1, \frac{\diam(X)}{6}\rbrace$ such that $B_R\subset \Omega''$. By \eqref{quasiboundedness}, we get 
$$
    \Vert u\Vert_{L^{\infty}(B_{\frac{R}{2}})}\leq C_R=C_R(\data, R, \Omega'', \Vert u\Vert_{N^{1,p}(\Omega'')})<\infty.
$$
    Since $\overline{\Omega'}$ is compact, we can cover it with a finite number of such balls, and conclude that
    $\Vert u\Vert_{L^{\infty}(\Omega')}<\infty$ 
    Therefore, $u$ is locally bounded.
\end{proof}

\section{Estimates for frozen functionals}\label{frozen}
In this section we collect the necessary regularity results for quasiminima of the so-called frozen functionals used in subsequent sections. We consider functionals of the type

\begin{equation}\label{eq1frozen}
\int_{\Omega}H_0(g_u)\dd\mu=\int_{\Omega}(g_u^p+a_0g_u^q)\dd\mu,
\end{equation}
where $a_0\geq 0$ is a constant.

In the next lemma we collect some useful properties, which follow from the seminal work of Lieberman \cite{Lieberman}.  Let $H_0(t)=t^p+a_0t^q$. We denote by $h_0=H_0'$. Therefore, $h_0(t)=pt^{p-1}+a_0qt^{q-1}$.
\begin{lemma}\label{lemmafrozen1} The following properties hold.
\begin{enumerate}
\item $H_0$ is $C^2$ and convex.
\item $\frac{t h_0(t)}{q}\leq H_0(t)\leq th_0(t)$, if $t\geq 0$.
\item $\frac{H_0(a)}{H_0(b)}\leq\frac{a}{b}$ if $b\geq a>0$.
\item $h_0(t)\leq h_0(2t)\leq 2^{q-1}h_0(t)$, if $t\geq 0$.
\item $ah_0(b)\leq ah_0(a)+b h_0(b)$, if $a,b>0$.
\end{enumerate}
\end{lemma}

\begin{definition} A function $u\in N_{\textrm{loc}}^{1,1}(\Omega)$ is a local quasiminimizer of \eqref{eq1frozen} on $\Omega$, for $K\geq 1$, if and only if $H_0(g_u)\in L_{\textrm{loc}}^{1}(\Omega)$ and the quasi-minimality condition
$$
\int_{\Omega'\cap\lbrace u\neq v\rbrace}H_0(g_u)\dd\mu\leq K\int_{\Omega'\cap\lbrace u\neq v\rbrace}H_0(g_v)\dd\mu
$$
holds for every $v\in N^{1,1}(\Omega')$, such that $u-v\in N^{1,1}_0(\Omega')$ and $H_0(g_u)\in L^1(\Omega')$, where $\Omega'$ is any open and bounded subset of $\Omega$ such that $\Omega'\Subset \Omega.$

\end{definition}

\subsection{Sobolev and Caccioppoli type inequalities}
We discuss a frozen Sobolev inequality for functions which vanish on a large set, see for example \cite{BB, KNP, KS}.

\begin{theorem}\label{FSPineq0largeSet} Assume that $ v\in N_{\textrm{loc}}^{1,1}(\Omega)$ with $H_0(g_v)\in L_{\textrm{loc}}^{1}(\Omega)$. Let $B_r$ be a ball. Assume that there exists $\gamma$, $0<\gamma<1$, such that
$$
\mu\left(\lbrace x\in B_r :\vert v(x)\vert>0\rbrace\right)\leq \gamma \mu(B_r).
$$
Then, there exist a constant $C= C(C_{PI},C_D, p,q,\lambda, \gamma)$, and exponents $0<d_2<1<d_1$, with $d_1=d_1(C_D, p,q)$ and $d_2=d_2(C_D, p,q)$ such that

$$
\left(\dashint_{B_r} H_0\left(\frac{v}{r}\right)^{d_1}\dd \mu\right)^{1/d_1}\leq C\left(\dashint_{B_{\lambda r}}H_0(g_v)^{d_2}\dd \mu\right)^{1/d_2},
$$
whenever $B_r\Subset\Omega$.
\end{theorem}

\begin{proof}
By Theorems \ref{kz} and \ref{sstars}, there exists $s=s(C_D, p,q)$ with $1<s<p\leq q<s^*$ such that $X$ supports a $(s^*,s)-$ Poincar\'e inequality. Then, by Lemma 2.1 of \cite{KS}, we obtain

$$
\left(\dashint_{B_r}\left\vert\frac{v}{r}\right\vert^{s^*}\dd\mu\right)^{\frac{1}{s^*}}\leq C\left(\dashint_{B_{\lambda r}}g_v^{s}\dd\mu\right)^{\frac{1}{s}}.
$$
Let $\frac{s}{p}<d_2<1$ and $1<d_1<\frac{s^*}{q}$. Since $pd_1\leq qd_1<s^*$ and $s<pd_2\leq qd_2$, by H\"older's inequality, we have
$$
\left(\dashint_{B_r}\left\vert\frac{v}{r}\right\vert^{pd_1}\dd\mu\right)^{\frac{1}{pd_1}}\leq C\left(\dashint_{B_{\lambda r}}g_v^{pd_2}\dd\mu\right)^{\frac{1}{pd_2}},
$$
and

$$
\left(\dashint_{B_r}\left\vert\frac{v}{r}\right\vert^{qd_1}\dd\mu\right)^{\frac{1}{qd_1}}\leq C\left(\dashint_{B_{\lambda r}}g_v^{qd_2}\dd\mu\right)^{\frac{1}{qd_2}}.
$$
It follows that

\begin{align*}
\left(\dashint_{B_r}H_0\left(\frac{v}{r}\right)^{d_1}\dd\mu\right)^{\frac{1}{d_1}}&=\left(\dashint_{B_r}\left(\left\vert\frac{v}{r}\right\vert^p+a_0\left\vert\frac{v}{r}\right\vert^q\right)^{d_1}\dd\mu\right)^{\frac{1}{d_1}}\\
&\leq \left(\dashint_{B_r}\left\vert\frac{v}{r}\right\vert^{pd_1}\dd\mu\right)^{\frac{1}{d_1}}+a_0\left(\dashint_{B_r}\left\vert\frac{v}{r}\right\vert^{qd_1}\dd\mu\right)^{\frac{1}{d_1}}\\
&\leq C \left( \left(\dashint_{B_{\lambda r}}g_v^{pd_2}\dd\mu\right)^{\frac{1}{d_2}}+a_0\left(\dashint_{B_{\lambda r}}g_v^{qd_2}\dd\mu\right)^{\frac{1}{d_2}}\right)\\
&\leq C\left( \dashint_{B_{\lambda r}}g_v^{pd_2}\dd\mu+\dashint_{B_{\lambda r}}(a_0g_v^q)^{d_2}\dd\mu\right)^{\frac{1}{d_2}}\\
&\leq C\left( \dashint_{B_{\lambda r}}(g_v^p+a_0g_v^q)^{d_2}\dd\mu\right)^{\frac{1}{d_2}}=\left(\dashint_{B_{\lambda r}}H_0\left(g_v\right)^{d_2}\dd\mu\right)^{\frac{1}{d_2}},
\end{align*}
where $C=C(C_D,C_{PI},p,q,\gamma,\lambda).$

Observe that all integrals are finite, since by H\"older 
$$\left(\dashint_{B_r}H_0\left(g_v\right)^{d_2}\dd\mu\right)^{\frac{1}{d_2}}\leq \dashint_{B_{\lambda r}}H_0\left(g_v\right)\dd\mu<\infty.
$$
\end{proof}

Another important Sobolev-type inequality is the following frozen Sobolev-Poincar\'e inequality for functions with zero boundary values, see for example \cite{BB}.

\begin{theorem}\label{FSPineq0boundary} Assume that $ v\in N_{0}^{1,1}(B_r)$ with $H_0(g_v)\in L^{1}(B_r)$, where $B_r$ is a ball in $\Omega$ with $0<r<\frac{\diam(X)}{4}$.
Then, there exist a constant $C= C(C_{PI},C_D, p,q,\lambda, \gamma)$, and exponents $0<d_2<1<d_1$, with $d_1=d_1(C_D, p,q)$ and $d_2=d_2(C_D, p,q)$ such that

$$
\left(\dashint_{B_r} H_0\left(\frac{v}{r}\right)^{d_1}\dd \mu\right)^{1/d_1}\leq C\left(\dashint_{B_{ r}}H_0(g_v)^{d_2}\dd \mu\right)^{1/d_2}.
$$
\end{theorem}
\begin{proof}This proof is analogous to the proof of Theorem \ref{FSPineq0largeSet}, but instead of using Lemma 2.1 of \cite{KS}, one uses Theorem 5.51 of \cite{BB}.
\end{proof}
Now we give an energy estimate for the frozen functional.
\begin{lemma}[Frozen Caccioppoli inequality]\label{FCaccioppoliIneq} 
Assume that  $u\in N^{1,1}_{\loc}({\Omega})$ with $H_0(g_u)\in L^1_{\loc}(\Omega)$  is a $K$-quasiminimizer in $\Omega$. Then there exists $C= C(K,q)>0$ such that for any choice of concentric balls $B_r\subset B_R\Subset\Omega $ and for all $k\in\mathbb{R}$, the following inequality 
\begin{align}\label{eq3frozen}\int_{B_r}H_0(g_{(u-k)_{\pm}})\, \dd \mu	& \leq C\int_{B_R} H_0\left(\frac{(u-k)_{\pm}}{R-r}\right)\, \dd\mu
\end{align}
is satisfied for $k\geq 0$. 
\end{lemma}

\begin{proof} The result is easily proven following the lines of the proof of Lemma 2.9 of \cite{KNP}. 
\end{proof}
\begin{definition}[Frozen De Giorgi Classes] We write $DG_{H_0}^{\pm}(\Omega)$ for the collection of functions  $u\in N^{1,1}_{\loc}({\Omega})$ with $H_0(g_u)\in L^1_{\loc}(\Omega)$ satisfying \eqref{eq3frozen} for any choice of concentric balls $B_r\subset B_R\Subset\Omega $ and for all $k\in\mathbb{R}$. The Frozen De Giorgi class is then defined as
$$
DG_{H_0}(\Omega)=DG_{H_0}^{+}(\Omega)\cap DG_{H_0}^{-}(\Omega).
$$
\end{definition}

\subsection{Oscillation reduction}
This section aims to prove an oscillation reduction result for frozen functionals. We start by providing some auxiliary results.

\begin{lemma}\label{lemma5frozen} Let $B_{6\lambda R}\Subset \Omega$ be a ball, such that $0<R\leq\min\lbrace 1, \frac{\diam (X)}{12\lambda}\rbrace$. 
Assume $u\in N^{1,1}_{\textrm{loc}}(\Omega)$, with $H(\cdot, g_u)\in L^1_{\textrm{loc}}(\Omega)$, satisfies $u\in DG_{H_0}^{+}(\Omega)$. Let $h>0$ such that the density condition
    \begin{equation}\label{eq17frozen}
    \mu(S_{h,R})\leq\gamma\mu(B_R),
    \end{equation}
    holds for some $\gamma\in(0,1).$ Then, there exist a positive constant $C=C(\gamma, C_{PI}, C_D,\lambda, q)$  and exponent $\xi=\xi(p,q,C_D)>0$ such that for any $k<h$ the following inequality
    \begin{align*}
    H_0\left(\frac{k-h}{R}\right)\mu(S_{k,R})&\leq C \left(\frac{\mu(S_{h,\lambda R})-\mu(S_{k,\lambda R})}{\mu(B_R)}\right)^{\xi}\left(\int_{S_{h,2\lambda R}}
H_0\left(\frac{u-h}{R}\right)\dd\mu\right)
\end{align*}
holds.
\end{lemma}

\begin{proof} 
For $k> h$, let us define
\begin{equation}\label{eq18frozen}
	v = \min\{u, k\} - \min\{u, h\}.
\end{equation} 
Note that, by \eqref{eq18frozen}, we have
\begin{equation} \label{eq19frozen}
	v=\begin{cases}
		0 \hspace{1.98cm} \mbox{if $u\leq h<k$,}\\
		u-h \hspace{1.3cm} \mbox{if $h<u<k$,}\\
		k-h \hspace{1.3cm} \mbox{if $u\geq k>h$.}
	\end{cases}
\end{equation}

From \eqref{eq17frozen} and \eqref{eq19frozen}, we deduce that $\mu(\{x \in B_R: v(x) > 0\}) \leq \gamma\mu(B_R).$ Then, by H\"older's inequality, the frozen Sobolev inequality for functions which vanish on a large set, Theorem \ref{FSPineq0largeSet}, and the definition of the De Giorgi class  $DG_{H_0}^{+}(\Omega)$, we obtain
\begin{align*}
     H_0\left(\frac{k-h}{R}\right)\mu(S_{k,R})&=\int_{S_{k,R}} H_0\left(\frac{v}{R}\right)\,\dd \mu\leq \int_{B_R} H_0\left(\frac{v}{R}\right)\,\dd\mu\leq\mu(B_R)\left(\dashint_{B_R}H_0\left(\frac{v}{R}\right)^{d_1}\dd\mu\right)^{\frac{1}{d_1}}\\
     &\leq C\mu(B_R)^{1-\frac{1}{d_2}}\left(\int_{S_{h,\lambda R}\setminus S_{k,\lambda R}}
H_0\left(g_{v}\right)^{d_2} \dd\mu\right)^{\frac{1}{d_2}}\\
&\leq C\left(\frac{\mu(S_{h,\lambda R})-\mu(S_{k,\lambda R})}{\mu(B_R)}\right)^{\frac{1}{d_2}-1}\int_{S_{h,\lambda R}\setminus S_{k,\lambda R}}
H_0\left(g_{(u-h)_+}\right) \dd\mu\\
&\leq C\left(\frac{\mu(S_{h,\lambda R})-\mu(S_{k,\lambda R})}{\mu(B_R)}\right)^{\frac{1}{d_2}-1}\int_{B_{\lambda R}}
H_0\left(g_{(u-h)_+}\right) \dd\mu\\
&\leq C\left(\frac{\mu(S_{h,\lambda R})-\mu(S_{k,\lambda R})}{\mu(B_R)}\right)^{\frac{1}{d_2}-1}\int_{B_{2\lambda R}}
H_0\left(\frac{(u-h)_+}{\lambda R}\right) \dd\mu\\
 &\leq C\left(\frac{\mu(S_{h,\lambda R})-\mu(S_{k,\lambda R})}{\mu(B_R)}\right)^{\frac{1}{d_2}-1}\int_{S_{h,2\lambda R}}
H_0\left(\frac{u-h}{R}\right) \dd\mu\\
&= C \left(\frac{\mu(S_{h,\lambda R})-\mu(S_{k,\lambda R})}{\mu(B_R)}\right)^{\xi}\left(\int_{S_{h,2\lambda R}}
H_0\left(\frac{u-h}{R}\right)\dd\mu\right),
\end{align*}
where $C=C(\gamma, C_{PI},C_D,\lambda, q)$ and $\xi=\xi(p,q,C_D)$.
\end{proof}

\noindent For any $B_\rho \Subset \Omega$, we shall denote
$$m(\rho) = \essinf_{B_\rho}u,\ \ M(\rho) = \esssup_{B_\rho}u$$
and
$$
\osc(u,\rho)=M(\rho)-m(\rho).
$$

\begin{lemma}\label{lemma6frozen} 
    Let $B_{6\lambda R}\Subset\Omega$ be a ball, with $0<R\leq\min\lbrace 1, \frac{\diam (X)}{12\lambda}\rbrace$. 
	Let $M =M(3\lambda R)$, $m = m(3 \lambda R)$.
Assume $u\in N^{1,1}_{\textrm{loc}}(\Omega)$ with $H(\cdot, g_u)\in L^1_{\textrm{loc}}(\Omega)$ satisfies $u\in DG_{H_0}^{+}(\Omega)$ and it is bounded from below. Moreover, assume that the density condition
 \begin{equation*}
     \mu(S_{k_0,R}) \leq \gamma \mu(B_R),\quad\textrm{for some }0 < \gamma < 1,
 \end{equation*}
 holds for $k_0= \frac{M+m}{2}$. Then there exists a positive constant $C=C(\gamma, C_{PI}, C_D, \lambda, q, Q)$ such that
 \begin{equation*}
    \dfrac{\mu(S_{k_n}, R)}{\mu(B_R)} \leq C n^{-\xi},
 \end{equation*} with $\xi=\xi(p,q,C_D)>0$, $n$ positive integer and $k_n=M-2^{-(n+1)} \osc(u,3\lambda R)$.
\end{lemma}
\begin{proof} 
 Let $k_j=M-2^{-(j+1)} (M-m)$, $j\in \mathbb{N} \cup \{0\}$. Therefore, $\lim_{j \to +\infty}k_j = M$. Note that $M-k_{j-1}=2^{-j} (M-m)$ and $k_j-k_{j-1}=2^{-(j+1)} (M-m)$. By Lemma \ref{lemma5frozen} and the doubling property, we deduce
	\begin{align*}
	H_0\left(\frac{2^{-(j+1)} (M-m)}{R}\right)\mu(S_{k_j,R})&	=H_0\left(\frac{(k_j-k_{j-1})}{R}\right)\mu(S_{k_j,R})\\ 
	&\leq C\left(\frac{\mu(S_{k_{j-1},\lambda R})-\mu(S_{k_{j},\lambda R})}{\mu(B_R)}\right)^{\xi}\int_{S_{k_{j-1},2 \lambda R}}
H_0\left(\frac{(u-k_{j-1})}{R}\right)\dd\mu\\
&\leq C\left(\frac{\mu(S_{k_{j-1},\lambda R})-\mu(S_{k_{j},\lambda R})}{\mu(B_R)}\right)^{\xi}\int_{B_{2 \lambda R}}
H_0\left(\frac{(M-k_{j-1})}{R}\right)\dd\mu\\
&\leq C\left(\frac{\mu(S_{k_{j-1},\lambda R})-\mu(S_{k_{j},\lambda R})}{\mu(B_R)}\right)^{\xi}\mu(B_R)H_0\left(\frac{2^{-j}(M-m)}{R}\right).
	\end{align*}

Therefore, for every $j\geq 0$, we obtain
	\begin{align*}
	\frac{\mu(S_{k_j,R})}{\mu(B_R)} \leq C\left(\frac{\mu(S_{k_{j-1},\lambda R})-\mu(S_{k_{j},\lambda R})}{\mu(B_R)}\right)^{\xi}.
	\end{align*}
 
\noindent If $n>j $, then $\mu(S_{k_n,  R})\leq \mu(S_{k_j,  R})$, and so
	\begin{align*}
		\left(\frac{\mu(S_{k_n,R})}{\mu(B_R)}\right)^{\frac{1}{\xi}} \leq C\left(\frac{\mu(S_{k_{j-1},\lambda R})-\mu(S_{k_{j},\lambda R})}{\mu(B_R)}\right).
	\end{align*}
	By summing the above inequality over $j=0,..., n-1$ and using the doubling property of the measure, we get
	\begin{align*}
		n\left(\frac{\mu(S_{k_n,R})}{\mu(B_R)}\right)^{\frac{1}{\xi}} \leq C\frac{\mu(B_{\lambda R})}{\mu(B_R)}\leq C=C(\gamma, C_{PI}, C_D,\lambda,q, Q).
	\end{align*}
 Therefore, 
 \begin{equation*}
    \dfrac{\mu(S_{k_n}, R)}{\mu(B_R)}\leq C n^{-\xi}.
 \end{equation*} 
 
 \end{proof}

\begin{lemma}\label{lemma7frozen} Let $B_{6\lambda R}\Subset\Omega$ be a ball, with $0<R\leq\min\lbrace 1, \frac{\diam (X)}{12\lambda}\rbrace$. 
	Assume $u\in N^{1,1}_{\textrm{loc}}(\Omega)$ with $H(\cdot, g_u)\in L^1_{\textrm{loc}}(\Omega)$ satisfies  $u\in DG_{H_0}^{+}(\Omega)$, and it is bounded from below. Then, for every $0< r\leq 3\lambda R$ and every $\kappa\in (0,1)$, there exists $\sigma\in (0,1)$, $\sigma=\sigma(C_{PI},C_D,p,q,\kappa)$, such that, if for some $\varepsilon>0$ the density condition
 \begin{equation}\label{eq20frozen}
     \mu\left(S_{M(r)-\varepsilon\osc(u,r), r}\right)\leq \sigma\mu(B_r)
 \end{equation}
holds, then
\begin{equation}\label{eq21frozen}
    u(x)\leq M(r)-\kappa\varepsilon\osc(u,r)
\end{equation}
holds a.e. in $B_{r/2}$.
 \end{lemma}

\begin{proof}
    Consider a sequence of nested balls $\lbrace B_{\rho_i}\rbrace$ concentric to $B_R$ for $i\geq 0$, where $\rho_i=\frac{r}{2}(1+2^{-i})\searrow\frac{r}{2}$, and define also $\bar{\rho_i}=(\rho_i+\rho_{i+1})/2$. We use Lipschitz cut-off functions $\eta_i\in \textrm{Lip}_0(B_{\bar{\rho_i}})$ such that $\eta_i= 1$ on $B_{\rho_i}$ and $g_{\eta_i}\leq\frac{2^{i}}{r}$. We also define the levels $k_i=M(r)-\kappa\varepsilon\osc(u,r)-(1-\kappa)\varepsilon\osc(u,r)/2^{i}\nearrow M(r)-\kappa\varepsilon\osc(u,r)$. Lastly, we define $v_i=\eta_i(u-k_i)_+$. By definition and Leibniz's rule, we obtain

$$
g_{v_i}\leq g_{(u-k_i)_+}\eta_i+(u-k_i)_+ g_{\eta_i}\leq g_{(u-k_i)_+}\eta_i+\frac{2^i}{r}(u-k_i)_+. 
$$

By Theorem \ref{FSPineq0boundary}, H\"older's inequality, the definition of $DG_{H_0}^{+}(\Omega)$ and the doubling property of our measure, we obtain
\begin{align*}
    H_0\left(\frac{(1-\kappa)\varepsilon\osc(u,r)}{2^{i+1}r}\right)\mu(S_{k_{i+1},\rho_{i+1}})&=H_0\left(\frac{k_{i+1}-k_i}{R}\right)\mu(S_{k_{i+1},\rho_{i+1}})\leq C\int_{S_{k_{i},\bar{\rho_{i}}}}H_0\left(\frac{v_i}{R}\right)\dd\mu\\
&\leq\mu(S_{k_{i},\bar{\rho_{i}}})^{1-\frac{1}{d_1}}\mu(B_{\bar{\rho_{i}}})^{\frac{1}{d_1}}\left(\dashint_{B_{\bar{\rho_{i}}}}H_0\left(\frac{v_i}{r}\right)^{d_1}\dd\mu \right)^{\frac{1}{d_1}}\\
    &\leq C\mu(S_{k_{i},\bar{\rho_{i}}})^{1-\frac{1}{d_1}}\mu(B_{\bar{\rho_{i}}})^{\frac{1}{d_1}}\left(\dashint_{B_{\bar{\rho_{i}}}}H_0\left(g_{v_i}\right)^{d_2}\dd\mu \right)^{\frac{1}{d_2}}  \\ 
    &\leq C\mu(S_{k_{i},\bar{\rho_{i}}})^{1-\frac{1}{d_1}}\mu(B_{\bar{\rho_{i}}})^{\frac{1}{d_1}}\dashint_{B_{\bar{\rho_{i}}}}H_0\left(g_{v_i}\right)\dd\mu\\
       &= C\left(\frac{\mu(S_{k_{i},\bar{\rho_{i}}})}{\mu(B_{\bar{\rho_{i}}})}\right)^{1-\frac{1}{d_1}} \int_{B_{\bar{\rho_i}}}H_0\left(g_{v_i}\right)\dd\mu\\
    &\leq C\left(\frac{\mu(S_{k_{i},\bar{\rho_{i}}})}{\mu(B_{\bar{\rho_{i}}})}\right)^{1-\frac{1}{d_1}} \Bigg(\int_{B_{\rho_i}}H_0\left(\frac{2^i(u-k_i)_+}{r}\right)\dd\mu\\
    &\qquad\qquad\qquad\qquad\qquad\qquad+\int_{B_{\bar{\rho_i}}}H_0\left(g_{(u-k_i)_+}\right)\dd\mu\Bigg)\\    
    &\leq C\left(\frac{\mu(S_{k_{i},\rho_{i}})}{\mu(B_{\rho_{i}})}\right)^{1-\frac{1}{d_1}} \Bigg(\int_{B_{\rho_i}}H_0\left(\frac{(u-k_i)_+}{(\rho_i-\bar{\rho_i})}\right)\dd\mu\\
   &\qquad\qquad\qquad\qquad\qquad\qquad +\int_{B_{\rho_i}}H_0\left(\frac{2^{i}}{r}(u-k_i)_+\right)\dd\mu\Bigg)\\
    &=C\left(\frac{\mu(S_{k_{i},\rho_{i}})}{\mu(B_{\rho_{i}})}\right)^{1-\frac{1}{d_1}} \Bigg(\int_{B_{\rho_i}}H_0\left(\frac{2^{i+3}}{r}(u-k_i)_+\right)\dd\mu\\
    &\qquad\qquad\qquad\qquad\qquad\qquad+\int_{B_{\rho_i}}H_0\left(\frac{2^{i}}{r}(u-k_i)_+\right)\dd\mu\Bigg).
\end{align*}
 Now, by Lemma \ref{lemmafrozen1} and induction we have the following inequalities involving $H_0$,
 \begin{equation}\label{frozeninduction1}
 H_0(2^{n}a)\leq 2^{nq}qH_0(a),\qquad\textrm{for }a\geq0\textrm{ and }n\in\mathbb{N}
 \end{equation}
 and
\begin{equation}\label{frozeninduction2}
 \frac{2^{-nq}}{q}H_0(b)\leq H_0(2^{-n}b),\qquad\textrm{for }b\geq0 \textrm{ and }n\in\mathbb{N}.
 \end{equation} 
 
Therefore, continuing the last set of inequalities we obtain
\begin{align*}
    H_0\left(\frac{(1-\kappa)\varepsilon\osc(u,r)}{2^{i+1}r}\right)\mu(S_{k_{i+1},\rho_{i+1}})&\leq C\left(\frac{\mu(S_{k_{i},\rho_{i}})}{\mu(B_{\rho_{i}})}\right)^{1-\frac{1}{d_1}}(2^{i})^{q}\left(\int_{B_{\rho_i}}H_0\left(\frac{(u-k_i)_+}{r}\right)\dd\mu\right)\\
    &\leq   C\left(\frac{\mu(S_{k_{i},\rho_{i}})}{\mu(B_{\rho_{i}})}\right)^{1-\frac{1}{d_1}} (2^{i})^{q}H_0\left(\frac{M(r)-k_i}{r}\right)\mu(S_{k_i, \rho_i})\\
    &\leq   C\left(\frac{\mu(S_{k_{i},\rho_{i}})}{\mu(B_{\rho_{i}})}\right)^{1-\frac{1}{d_1}}(2^{i})^{q}H_0\left(\frac{M(r)-k_0}{r}\right)\mu(S_{k_i, \rho_i})\\
     &\leq   C\left(\frac{\mu(S_{k_{i},\rho_{i}})}{\mu(B_{\rho_{i}})}\right)^{1-\frac{1}{d_1}} (2^{i})^{q}H_0\left(\frac{\varepsilon\osc(u,r)}{r}\right)\mu(S_{k_i, \rho_i}).
\end{align*}
By \eqref{frozeninduction2} and the definition of $H_0$,  we then obtain
$$
 \frac{2^{-i-1}q}{q}(1-\kappa)^{q}H_0\left(\frac{\varepsilon\osc(u,r)}{r}\right)\mu(S_{k_{i+1},\rho_{i+1}})\leq   C\left(\frac{\mu(S_{k_{i},\rho_{i}})}{\mu(B_{\rho_{i}})}\right)^{1-\frac{1}{d_1}}(2^{i})^{q}H_0\left(\frac{\varepsilon\osc(u,r)}{r}\right)\mu(S_{k_i, \rho_i}).
$$
Therefore,
$$
\frac{\mu(S_{k_{i+1},\rho_{i+1}})}{\mu(B_{\rho_{i+1}})}\leq \frac{\mu(S_{k_{i+1},\rho_{i+1}})}{\mu(B_{\rho_{i}})}\leq C \left(\frac{\mu(S_{k_{i},\rho_{i}})}{\mu(B_{\rho_{i}})}\right)^{1+\frac{d_1-1}{d_1}}(2^{q+1})^i(1-\kappa)^{-q},
$$
where $C=C(C_{PI}, C_D,p,q)$. Therefore, if we define $\Sigma_i=\frac{\mu(S_{k_{i},\rho_{i}})}{\mu(B_{\rho_{i}})}$, we obtain the recursive estimate
$$
\Sigma_{i+1}\leq \frac{C(2^{q+1})^{i}}{(1-\kappa)^{-q}}\Sigma_{i}^{1+\frac{d_1-1}{d_1}}, \qquad C>0.
$$
In order to prove \eqref{eq21frozen}, we need $\Sigma_i\rightarrow 0$ as $i\rightarrow +\infty$. By Lemma 7.1 in \cite{G}, this happens if 
$$
\Sigma_0\leq\left(\frac{C}{(1-\kappa)^q}\right)^{-\frac{d_1}{d_1-1}}(2^{q+1})^{-\left(\frac{d_1}{d_1-1}\right)^{2}}.
$$
Meaning, 
$$
\frac{\mu(S_{M(r)-\varepsilon\osc(u,r),r})}{\mu(B_r)}<\sigma\leq\left(\frac{C}{(1-\kappa)^q}\right)^{-\frac{d_1}{d_1-1}}(2^{q+1})^{-\left(\frac{d_1}{d_1-1}\right)^{2}},
$$
which is equation \eqref{eq20frozen}.

Therefore, by choosing $\sigma=\sigma(C_{PI},C_D,p,q,\kappa)$, small enough, we indeed obtain \eqref{eq21frozen} as wanted.

\end{proof}

Now, we are ready to prove the oscillation reduction result for the frozen functional. 

\begin{lemma}\label{lemma8frozen} Let $B_{6\lambda R}\Subset\Omega$ be a ball, with $0<R\leq\min\lbrace 1, \frac{\diam (X)}{12\lambda}\rbrace$. 
	Assume $u\in N^{1,1}_{\textrm{loc}}(\Omega)$ with $H(\cdot, g_u)\in L^1_{\textrm{loc}}(\Omega)$ satisfies  $u\in DG_{H_0}^{+}(\Omega)$, and it is bounded from below. Let $0<\rho<R$. Then, there exists $0 < \eta < 1$ such that
	$$\osc(u, 3\lambda \rho) \leq 4^\eta
	\left(\frac{\rho }{R}\right)^\eta \osc(u, 3\lambda R),$$
 with $\eta=\eta(C_{PI},C_D,p,q)$.
\end{lemma}

\begin{proof}
	We consider $M$ and $m$ as in Lemma \ref{lemma7frozen}. By Lemma \ref{lemma7frozen}, for $\kappa=\frac{1}{2}$, there exists $\sigma\in(0,1)$ such that if 

 $$
 \mu(S_{M-\varepsilon(M-m), 3\lambda R})\leq \sigma\mu(B_{3\lambda R}),
 $$
 for some $\varepsilon>0$, then
 $$
 u(x)\leq M-\frac{\varepsilon}{2}(M-m),
 $$
$\mu$-a.e. in $B_{\frac{3\lambda R}{2}}$. This would imply
 $$
 M\left(\frac{3\lambda R}{2}\right)=\esssup_{B_{\frac{3\lambda R}{2}}} u\leq M-\frac{\varepsilon}{2}(M-m).
 $$
 Now, 
 $$
 -m\left(\frac{3\lambda R}{2}\right)=-\essinf_{B_{\frac{3\lambda R}{2}}} u\leq -\essinf_{B_{3\lambda R}}u=-m.
 $$
 Adding these two inequalities gives
 \begin{equation}\label{eq11noteslocal}
     \osc\left(u,\frac{3\lambda R}{2}\right)\leq \left(1-\frac{\varepsilon}{2}\right)\osc(u, 3\lambda R).
 \end{equation}
So, all we are left to do is determine $\varepsilon>0$. If we define $k_n=M-\frac{1}{2^{n+1}}(M-m)$. By Lemma \ref{lemma7frozen} and the doubling property we have
$$
\frac{\mu(S_{k_n, 3\lambda R})}{\mu(B_{3\lambda R})}\leq C n^{-\xi}\rightarrow 0, \qquad\textrm{when }n\rightarrow +\infty.
$$
 Therefore, by choosing $n^*=n^*(\xi,\sigma)$ such that
 $$
 Cn^{-\xi}<\sigma,\qquad\textrm{for all }n\geq n^{*},
 $$
 we then have that the $\varepsilon>0$ we were looking for is $\varepsilon=2^{-(n^*+1)}$, $\varepsilon=\varepsilon(C_{PI},C_D,p,q)$. Therefore, by \eqref{eq11noteslocal}, we obtain
 \begin{equation}\label{eq12noteslocal}
     \osc\left(u,\frac{3\lambda R}{2}\right)<\tau \osc(u,3\lambda R)
 \end{equation}
where $\tau=1-2^{-(n^*+2)}=\tau(C_{PI},C_D,p,q)$.

	Now, we consider an index $j \geq 1$ such that $$4^{j-1} \leq\frac{R}{\rho}  < 4^j.$$ Then, from inequality \eqref{eq12noteslocal}, we get
	$$\osc(u, 3\lambda \rho) \leq \tau^{j-1} \osc(u, 3\lambda  4^{j-1}\rho) \leq \tau^{j-1} \osc(u,3 \lambda R).$$
	We observe that $\tau= 4^{\log_4 \tau}= 4^{\frac{\log \tau}{\log 4}}$ and we deduce that
	$$ \tau^{j-1}=4^{\frac{\log \tau}{log 4}(j-1)}=4^{-(j-1)\left(-\frac{\log \tau}{log 4}\right)}=\left(\frac{4}{4^j}\right)^{-\frac{\log \tau}{\log 4}}\leq\left(\frac{4}{\frac{R}{\rho}}\right)^{-\frac{\log \tau}{\log 4}}= 4^{\eta} \left(\frac{R}{\rho}\right)^{-\eta} ,$$ where $\eta=-\frac{\log\tau}{\log 4}< 1$, $\eta=\eta(C_{PI},C_D,p,q)$. At the end, we obtain
	$$\osc(u, 3\lambda \rho)) \leq 4^{\eta} \left(\frac{R}{\rho}\right)^{-\eta} \osc(u, 3\lambda R) ,$$ that completes the proof. 
\end{proof}

\subsection{Weak Harnack inequalities}
This section aims to prove some weak Harnack inequalities for frozen functionals. 

\begin{theorem}\label{upperharnackineqfrozen} Let $B_{6\lambda R}\Subset \Omega$ be a ball, such that $0<R\leq\min\lbrace 1, \frac{\diam (X)}{12\lambda}\rbrace$. Assume $u\in DG_{H_0}(\Omega)$. Let $t_+>0$ and $B_{\rho}\subset B_R$, $0<\rho<R$, then there is $C=C(C_{PI},C_D, p,q,t_+,Q)$ such that

$$
\esssup_{B_{\rho}} u\leq \frac{C}{(1-\frac{\rho}{R})^{\frac{Q}{t_+}}}\left(\dashint_{B_R}u^{t_+}\dd\mu\right)^{\frac{1}{t_+}}.
$$
\end{theorem}
\begin{proof}
Assume $u\in DG_{H_0}^{+}(\Omega)$, the other case is handled similarly. For $k>0$ to be chosen later and any positive integer $n$, we set
$$
k_n=k(1-2^{-n})\nearrow k,\qquad\qquad R_n=(1+2^{-n})R/2\searrow R/2.
$$
and 

$$
\bar{R_n}=\frac{1}{2}(R_n+R_{n+1}),\qquad\qquad \rho_n=\frac{R/2}{2^{n+2}}=R_n-\bar{R_n}.
$$
We grab cut-off functions $\eta_n\in \textrm{Lip}_0(B_{\bar{R_n}})$ with compact support in $B_{\bar{R_n}}$, $\eta_n= 1$ on $B_{R_{n+1}}$ and $0\leq \eta_n\leq 1$, $g_{\eta_n}\leq\frac{2^{n+2}}{R/2}=\frac{1}{\rho_n}.$

First, we notice the following. By Lemma \ref{lemmafrozen1} and \eqref{frozeninduction1} we have

\begin{align*}
H_0\left(\frac{(u-k_{n+1})_+}{\rho_n}\right)&=H_0\left(\frac{2^{n+2}(u-k_{n+1})_+}{R/2}\right)\\
&\leq (2^{n+2})^{q}qH_0\left(\frac{(u-k_{n+1})_+}{R/2}\right)\\
&=(2^{n+2})^{q}qH_0\left(\frac{4(u-k_{n+1})_+}{4R/2}\right)\\
&\leq C2^{nq}H_0\left(\frac{(u-k_{n+1})_+}{4R/2}\right)\\
&\leq C2^{nq}H_0\left(\frac{(u-k_{n+1})_+}{\bar{R_n}}\right).
\end{align*}
Therefore, by H\"older's inequality, Theorem \ref{FSPineq0boundary}, the doubling property of the measure, Leibniz rule and the definition of $ DG_{H_0}^{+}(\Omega)$, we obtain

\begin{align*}
\int_{B_{R_{n+1}}}H_0\left(\frac{(u-k_{n+1})_+}{\rho_n}\right)\dd\mu&\leq C2^{nq} \int_{B_{R_{n+1}}}H_0\left(\frac{(u-k_{n+1})_+}{\bar{R_n}}\right)\dd\mu\\
&\leq  C2^{nq} \int_{B_{\bar{R_{n}}}}H_0\left(\frac{\eta_n(u-k_{n+1})_+}{\bar{R_n}}\right)\dd\mu\\
&\leq  C2^{nq} \mu(S_{k_{n+1},R_n})^{1-\frac{1}{d_1}}\left(\int_{S_{k_{n+1},R_n}}H_0\left(\frac{\eta_n(u-k_{n+1})_+}{\bar{R_n}}\right)^{d_1}\dd\mu\right)^{\frac{1}{d_1}}\\
&\leq  C2^{nq} \mu(S_{k_{n+1},R_n})^{1-\frac{1}{d_1}}\mu(B_{\bar{R_n}})^{\frac{1}{d_1}}\left(\dashint_{B_{\bar{R_n}}}H_0\left(\frac{\eta_n(u-k_{n+1})_+}{\bar{R_n}}\right)^{d_1}\dd\mu\right)^{\frac{1}{d_1}}\\
&\leq  C2^{nq} \mu(S_{k_{n+1},R_n})^{1-\frac{1}{d_1}}\mu(B_{\bar{R_n}})^{\frac{1}{d_1}}\left(\dashint_{B_{\bar{R_n}}}H_0\left(g_{\eta_n(u-k_{n+1})_+}\right)^{d_2}\dd\mu\right)^{\frac{1}{d_2}}\\
&\leq  C2^{nq} \left(\frac{\mu(S_{k_{n+1},R_n})}{\mu(B_{\bar{R_n}})}\right)^{1-\frac{1}{d_1}}\int_{B_{\bar{R_n}}}H_0\left(g_{\eta_n(u-k_{n+1})_+}\right)\dd\mu\\
&\leq C2^{nq} \left(\frac{\mu(S_{k_{n+1},R_n})}{\mu(B_{R})}\right)^{1-\frac{1}{d_1}}\Bigg(\int_{B_{\bar{R_n}}}H_0\left(g_{(u-k_{n+1})_+}\right)\dd\mu\\
&\qquad\qquad\qquad\qquad\qquad\qquad\qquad\qquad+\int_{B_{\bar{R_n}}}H_0\left(\frac{(u-k_{n+1})_+}{\rho_n}\right)\dd\mu\Bigg)\\
&\leq C2^{nq} \left(\frac{\mu(S_{k_{n+1},R_n})}{\mu(B_{R})}\right)^{1-\frac{1}{d_1}}\int_{B_{R_n}}H_0\left(\frac{(u-k_{n+1})_+}{\rho_n}\right)\dd\mu\\
&\leq C2^{nq} \left(\frac{\mu(S_{k_{n+1},R_n})}{\mu(B_{R})}\right)^{1-\frac{1}{d_1}}\int_{B_{R_n}}H_0\left(\frac{(u-k_{n})_+}{\rho_n}\right)\dd\mu.
\end{align*}
Since $H_0$ is increasing, by the definition of $\rho_n$, and \eqref{frozeninduction1}, we obtain
\begin{align*}
\int_{B_{R_{n+1}}}H_0\left(\frac{(u-k_{n+1})_+}{R/2}\right)\dd\mu&\leq\int_{B_{R_{n+1}}}H_0\left(\frac{2^{n+2}(u-k_{n+1})_+}{R/2}\right)\dd\mu\\
&\leq C2^{nq} \left(\frac{\mu(S_{k_{n+1},R_n})}{\mu(B_{R})}\right)^{1-\frac{1}{d_1}}\int_{B_{R_n}}H_0\left(\frac{2^{n+2}(u-k_{n})_+}{R/2}\right)\dd\mu\\
&\leq C2^{nq} \left(\frac{\mu(S_{k_{n+1},R_n})}{\mu(B_{R})}\right)^{1-\frac{1}{d_1}}(2^{n+2})^{q}q\int_{B_{R_n}}H_0\left(\frac{(u-k_{n})_+}{R/2}\right)\dd\mu\\
&\leq C2^{(q+1)n} \left(\frac{\mu(S_{k_{n+1},R_n})}{\mu(B_{R})}\right)^{1-\frac{1}{d_1}}\int_{B_{R_n}}H_0\left(\frac{(u-k_{n})_+}{R/2}\right)\dd\mu
\end{align*}
Then,

\begin{equation}\label{eq4frozen}
\int_{B_{R_{n+1}}}H_0\left(\frac{(u-k_{n+1})_+}{R/2}\right)\dd\mu\leq C2^{(q+1)n} \left(\frac{\mu(S_{k_{n+1},R_n})}{\mu(B_{R})}\right)^{1-\frac{1}{d_1}}\int_{B_{R_n}}H_0\left(\frac{(u-k_{n})_+}{R/2}\right)\dd\mu,
\end{equation}

where $C=C(C_D,C_{PI},\lambda,p,q,Q)$.

Now, we set 
$$
Y_n=H_0\left(\frac{k}{R}\right)^{-1}R^{-\frac{d_1}{d_1-1}}\int_{B_{R_n}}H_0\left(\frac{(u-k_{n})_+}{R/2}\right)\dd\mu.
$$
By Lemma \ref{lemmafrozen1}, we have the following inequality
$$
\frac{H_0\left(\frac{2k}{R/2}\right)}{H_0\left(\frac{k}{R}\right)}\geq\frac{2}{1/2}=4.
$$
Therefore,
\begin{equation}\label{eq5frozen}
4\mu(S_{k_{n+1},R_n})\leq \frac{H_0\left(\frac{2k}{R/2}\right)}{H_0\left(\frac{k}{R}\right)}\mu(S_{k_{n+1},R_n}).
\end{equation}
On the other hand, again by Lemma \ref{lemmafrozen1} and \eqref{frozeninduction2}, we have
\begin{align*}
H_0\left(\frac{2k}{R/2}\right)&\leq 4^{q}(2^{q})^{n}q H_0\left(\frac{2^{-n-1}k}{R/2}\right)\\
&= C2^{qn}H_0\left(\frac{k_{n+1}-k_n}{R/2}\right).
\end{align*}
Therefore,
\begin{align*}
H_0\left(\frac{2k}{R/2}\right)\mu(S_{k_{n+1},R_n})&\leq C2^{qn}H_0\left(\frac{k_{n+1}-k_n}{R/2}\right)\mu(S_{k_{n+1},R_n})\\
&\leq C2^{qn}\int_{S_{k_{n+1},R_n}}H_0\left(\frac{u-k_n}{R/2}\right)\dd\mu\\
&= C2^{qn}\int_{B_{R_n}}H_0\left(\frac{(u-k_n)_+}{R/2}\right)\dd\mu\\
&=C2^{qn}H_0\left(\frac{k}{R}\right)R^{\frac{d_1}{d_1-1}}Y_n
\end{align*}
So,
\begin{equation}\label{eq6frozen}
\frac{H_0\left(\frac{2k}{R/2}\right)}{H_0\left(\frac{k}{R}\right)}\mu(S_{k_{n+1},R_n})\leq C2^{qn}H_0\left(\frac{k}{R}\right)R^{\frac{d_1}{d_1-1}}Y_n
\end{equation}
By \eqref{eq5frozen} and \eqref{eq6frozen}, we have
$$
\mu(S_{k_{n+1},R_n})\leq C2^{qn}H_0\left(\frac{k}{R}\right)R^{\frac{d_1}{d_1-1}}Y_n.
$$
Furthermore, by \eqref{eq4frozen} and this last inequality, we deduce
\begin{align*}
H_0\left(\frac{k}{R}\right)R^{\frac{d_1}{d_1-1}}Y_{n+1}&\leq C\left(\frac{C2^{qn}Y_nR^{\frac{d_1}{d_1-1}}}{\mu(B_R)}\right)^{1-\frac{1}{d_1}}2^{(q+1)n}H_0\left(\frac{k}{R}\right)R^{\frac{d_1}{d_1-1}}Y_n\\
&\leq C 2^{(q+1)n}\mu(B_R)^{\frac{1}{d_1}-1}RH_0\left(\frac{k}{R}\right)R^{\frac{d_1}{d_1-1}}Y_n^{1+\frac{d_1-1}{d_1}}.
\end{align*}
Therefore,
$$
Y_{n+1}\leq \left(C\mu(B_R)^{\frac{1}{d_1}-1}R\right)(2^{q+1})^{n}Y_n^{1+\frac{d_1-1}{d_1}}.
$$
 By Lemma 7.1 in \cite{G}, we have that $Y_n\rightarrow 0$ provided
 $$
 Y_0\leq \left(C\mu(B_R)^{\frac{1}{d_1}-1}R\right)^{-\frac{d_1}{d_1-1}}(2^{q+1})^{-\left(\frac{d_1}{d_1-1}\right)^{2}}= C\mu(B_R)R^{-\frac{d_1}{d_1-1}},
 $$
where $C=C(C_{PI},C_D,p,q,Q)$.
Meaning, it is enough to choose $k$, such that
$$
H_0\left(\frac{k}{R}\right)\geq C\dashint_{B_R}H_0\left(\frac{(u-k_n)_+}{R/2}\right)\dd\mu.
$$
So, we have
$$
\lim_{n\rightarrow\infty}Y_n=0
$$
Then,
$$
\lim_{n\rightarrow\infty}\int_{B_{R_n}}H_0\left(\frac{(u-k_n)_+}{R/2}\right)\dd\mu=\int_{B_{R/2}}H_0\left(\frac{(u-k)_+}{R/2}\right)\dd\mu=0.
$$
Therefore, 
$$
\esssup_{B_{R/2}}\left(\frac{u_+}{R}\right)\leq \frac{k}{R}.
$$
Since $H_0$ is increasing, we then obtain
$$
H_0\left(\esssup_{B_{R/2}}\left(\frac{u_+}{R}\right)\right)\leq H_0\left(\frac{k}{R}\right)=C\dashint_{B_R}H_0\left(\frac{u_+}{R}\right)\dd\mu.
$$
We now define $G(t)=H_0(t^{\frac{1}{q}})$. Then by Lemma \ref{lemmafrozen1}
$$
G''(t)=\frac{1}{q}t^{\frac{1}{q}-2}h_0(t^{\frac{1}{q}})\left(\frac{1}{q}-1+\frac{1}{q}\frac{h_0'(t^{\frac{1}{q}})t^{\frac{1}{q}}}{h_0(t^{\frac{1}{q}})}\right)\leq 0
$$
Hence, Jensen's inequality implies that
$$
\dashint_{B_R}H_0\left(\frac{u_+}{R}\right)\dd\mu\leq H_0\left(\left(\dashint_{B_R}\left(\frac{u_+}{R}\right)^{q}\dd\mu\right)^{\frac{1}{q}}\right).
$$
So,
$$
H_0\left(\esssup_{B_{R/2}}\left(\frac{u_+}{R}\right)\right)\leq C H_0\left(\left(\dashint_{B_R}\left(\frac{u_+}{R}\right)^{q}\dd\mu\right)^{\frac{1}{q}}\right).
$$
Since $H_0$ is increasing, we can conclude
\begin{equation}\label{eq8frozen}
\esssup_{B_{R/2}}u_+\leq C \left(\dashint_{B_R}u_+^{q}\dd\mu\right)^{\frac{1}{q}},
\end{equation}

where $C=C(C_{PI},C_D,p,q,Q).$
Now, it is easy to see that there is nothing particular in the factor $1/2$ in the radius of the last inequality. Indeed, if $0<\rho<r\leq R$, then
\begin{equation}\label{eq9frozen}
\esssup_{B_{\rho}}u_+\leq \frac{C}{(1-\frac{\rho}{r})^{\frac{Q}{q}}} \left(\dashint_{B_r}u_+^{q}\dd\mu\right)^{\frac{1}{q}}.
\end{equation}
To see this, let $\varepsilon>0$ and take $y\in B_{\rho}$ so that $u_+(y)^{q}\geq\left(\esssup_{B_{\rho}} u_+\right)^{q}-\varepsilon$.
Then, by \eqref{eq8frozen}
\begin{align*}
\left(\esssup_{B_{\rho}} u_+\right)^{q}&\leq u_+(y)^{q}+\varepsilon\leq\left(\esssup_{B_{\rho}} u_+\right)^{q}+\varepsilon\\
&\leq C\left(\dashint_{B_{\frac{r-\rho}{2}}}u_+^{q}\dd\mu\right).
\end{align*}
The doubling property of the measure $\mu$ implies that
$$
\mu\left(B_{\frac{r-\rho}{2}}\right)\geq C\left(1-\frac{\rho}{r}\right)^{Q}\mu(B_r),
$$
from which the claim follows.

Finally, we are left to see that \eqref{eq9frozen} holds for every exponent $t_+>0$. To be more precise, for $t_+>0$ there is a constant $C$ such that

$$
\esssup_{B_{\rho}} u_+\leq \frac{C}{(1-\frac{\rho}{R})^{\frac{Q}{t_+}}}\left(\dashint_{B_R}u^{t_+}\dd\mu\right)^{\frac{1}{t_+}},
$$
when $0<\rho<R<\infty$. If $t_+>q$, the claim follows directly from H\"older's inequality. Suppose that $0<t_+<q$, and let $0<\rho<r\leq R$. Then an application of Young's inequality and the doubling property of the measure gives

\begin{align*}
\esssup_{B_{\rho}} u_+&\leq \frac{C}{(1-\frac{\rho}{R})^{\frac{Q}{q}}}\left(\dashint_{B_r}(u^{t_+})(u_+^{q-t_+})\dd\mu\right)^{\frac{1}{q}}\\
&\leq \frac{C}{(1-\frac{\rho}{R})^{\frac{Q}{q}}}\left(\dashint_{B_r}u^{t_+}\dd\mu\right)^{\frac{1}{q}}\left(\esssup_{B_{r}} u_+\right)^{1-\frac{t_+}{q}}\\
&\leq \frac{t_+}{q}\frac{C}{(1-\frac{\rho}{R})^{\frac{Q}{t_+}}}\left(\dashint_{B_r}u^{t_+}\dd\mu\right)^{\frac{1}{t_+}}+\left(\frac{q-t_+}{q}\right)\esssup_{B_{r}} u_+\\
&=\varepsilon\esssup_{B_{r}} u_+\frac{C}{(1-\frac{\rho}{R})^{\frac{Q}{t_+}}}\left(\dashint_{B_r}u^{t_+}\dd\mu\right)^{\frac{1}{t_+}}\\
&\leq\varepsilon\esssup_{B_{r}} u_+\frac{C}{(r-\rho)^{\frac{Q}{t_+}}}\left(R^{Q}\dashint_{B_R}u^{t_+}\dd\mu\right)^{\frac{1}{t_+}},
\end{align*}
with $C=C(C_{PI},C_D,p,q,t_+,Q).$ Then, by Lemma 3.2 of \cite{KS}, there is a constant $C=C(q,Q,t_+)$, such that 
$$
\esssup_{B_{\rho}} u_+\leq \frac{C}{(1-\frac{\rho}{R})^{\frac{Q}{t_+}}}\left(\dashint_{B_R}u^{t_+}\dd\mu\right)^{\frac{1}{t_+}},
$$
as wanted.

\end{proof}

\begin{lemma}\label{lemma3frozen}
Let $B_{6\lambda R}\Subset \Omega$ be a ball, such that $0<R\leq\min\lbrace 1, \frac{\diam (X)}{12\lambda}\rbrace$. Assume $u\in DG_{H_0}^{-}(\Omega)$ is non negative. Fix $\delta\in (0,1)$ and suppose that
\begin{equation}\label{eq10frozen}
\frac{\mu(B_R\cap\lbrace u>l\rbrace)}{\mu(B_R)}\geq \delta
\end{equation} 
holds for some positive level $l>0$. Then there exist an exponent $\xi=\xi(C_D,p,q)>0$ and a constant $C=C(C_{PI},C_D,\lambda,\delta,q)$, such that
\begin{equation}\label{eq11frozen}
\frac{\mu(B_R\cap\lbrace u\leq 2^{-\hat{j}}l\rbrace)}{\mu(B_R)}\leq C\hat{j}^{-\xi}
\end{equation}
holds for any $\hat{j}\in\mathbb{N}$.
\end{lemma}

\begin{proof}
We define
\begin{equation}\label{eq12frozen}
v_j=\max\lbrace u, 2^{-j}l\rbrace-\max\lbrace u, 2^{-(j+1)}l\rbrace.
\end{equation}

Note that, by \eqref{eq12frozen}, we have
\begin{equation} \label{eq13frozen}
	v=\begin{cases}
		2^{-j}l-2^{-(j+1)}l=2^{-(j+1)}l, \hspace{2.15cm} \mbox{if $u\leq 2^{-(j+1)}l<2^{-j}l$,}\\
		2^{-j}l-u, \hspace{4.9cm} \mbox{if $2^{-(j+1)}l<u<2^{-j}l$,}\\
		0, \hspace{6cm} \mbox{if $u\geq 2^{-j}l>2^{-(j+1)}l$.}
	\end{cases}
\end{equation}
From \eqref{eq10frozen} and \eqref{eq13frozen}, we deduce that

\begin{align*}
\mu(\lbrace x\in B_R:v_j(x)>0\rbrace)&\leq \mu(\lbrace x\in B_R:u(x)\leq 2^{-j}l\rbrace)\\
&\leq \mu(\lbrace x\in B_R:u(x)\leq l\rbrace)\\
&\leq \mu(B_R)-\delta\mu(B_R)\\
&=(1-\delta)\mu(B_R),
\end{align*}
where, $0<1-\delta<1$.

Therefore, by Theorem \ref{FSPineq0largeSet}, H\"older's inequality and the definition of $DG_{H_0}^{-}(\Omega)$, we obtain
\begin{align*}
H_0\left(\frac{2^{-(j+1)}l}{R}\right)\mu(B_R&\cap\lbrace u\leq 2^{-(j+1)}l\rbrace)=\int_{B_R\cap\lbrace u\leq 2^{-(j+1)}l\rbrace}H_0\left(\frac{v_j}{R}\right)\dd\mu\leq \int_{B_R}H_0\left(\frac{v_j}{R}\right)\dd\mu\\
&\leq\mu(B_R)\left(\dashint_{B_R}H_0\left(\frac{v_j}{R}\right)^{d_1}\dd\mu\right)^{\frac{1}{d_1}}\\
&\leq C\mu(B_R)\left(\dashint_{B_{\lambda R}}H_0\left(g_{v_j}\right)^{d_2}\dd\mu\right)^{\frac{1}{d_2}}\\
&\leq C\mu(B_R)^{1-\frac{1}{d_2}}\left(\int_{B_{\lambda R}\cap\lbrace 2^{-(j+1)}l<u\leq 2^{-j}l\rbrace}H_0\left(g_{v_j}\right)^{d_2}\dd\mu\right)^{\frac{1}{d_2}}\\
&\leq C\left(\frac{\mu(B_{\lambda R}\cap\lbrace 2^{-(j+1)}l<u\leq 2^{-j}l\rbrace)}{\mu(B_R)}\right)^{\frac{1}{d_2}-1}\\
&\qquad\qquad\qquad\qquad\qquad\qquad\int_{B_{\lambda R}\cap\lbrace 2^{-(j+1)}l<u\leq 2^{-j}l\rbrace}H_0\left(g_{v_j}\right)\dd\mu\\
&\leq C\left(\frac{\mu(B_{\lambda R}\cap\lbrace 2^{-(j+1)}l<u\leq 2^{-j}l\rbrace)}{\mu(B_R)}\right)^{\frac{1}{d_2}-1}\int_{B_{\lambda R}}H_0\left(g_{(u-2^{-j}l)_-}\right)\dd\mu\\
&\leq C\left(\frac{\mu(B_{\lambda R}\cap\lbrace 2^{-(j+1)}l<u\leq 2^{-j}l\rbrace)}{\mu(B_R)}\right)^{\frac{1}{d_2}-1}\int_{B_{2\lambda R}}H_0\left(\frac{(u-2^{-j}l)_-}{\lambda R}\right)\dd\mu\\
&\leq C\left(\frac{\mu(B_{\lambda R}\cap\lbrace 2^{-(j+1)}l<u\leq 2^{-j}l\rbrace)}{\mu(B_R)}\right)^{\frac{1}{d_2}-1}\int_{B_{2\lambda R}\cap\lbrace u\leq 2^{-j}l\rbrace}H_0\left(\frac{2^{-j}l}{R}\right)\dd\mu\\
&\leq C\left(\frac{\mu(B_{\lambda R}\cap\lbrace 2^{-(j+1)}l<u\leq 2^{-j}l\rbrace)}{\mu(B_R)}\right)^{\frac{1}{d_2}-1}\mu(B_R)H_0\left(\frac{2^{-j}l}{R}\right).
\end{align*}
Therefore,
$$
\left(\frac{\mu(B_{R}\cap\lbrace u\leq 2^{-(j+1)}l\rbrace)}{\mu(B_R)}\right)\leq C\left(\frac{\mu(B_{\lambda R}\cap\lbrace 2^{-(j+1)}l<u\leq 2^{-j}l\rbrace)}{\mu(B_R)}\right)^{\frac{1}{d_2}-1}.
$$
If $\hat{j}>j$, then $\mu(B_{R}\cap\lbrace u\leq 2^{-\hat{j}}l\rbrace)\leq\mu(B_{R}\cap\lbrace u\leq 2^{-(j+1)}l\rbrace)$. Therefore,
$$
\left(\frac{\mu(B_{R}\cap\lbrace u\leq 2^{-\hat{j}}l\rbrace)}{\mu(B_R)}\right)^{\frac{d_2}{1-d_2}}\leq C\left(\frac{\mu(B_{\lambda R}\cap\lbrace 2^{-(j+1)}l<u\leq 2^{-j}l\rbrace)}{\mu(B_R)}\right).
$$
Summing up for $j=0,1,...,\hat{j}-1$, $\hat{j}\in\mathbb{N}$, and the doubling property of the measure, gives
\begin{align*}
\hat{j}\left(\frac{\mu(B_{R}\cap\lbrace u\leq 2^{-\hat{j}}l\rbrace)}{\mu(B_R)}\right)^{\frac{d_2}{1-d_2}}&\leq C\mu(B_R)^{-1}\sum_{j=0}^{\hat{j}-1}\mu(B_{\lambda R}\cap\lbrace 2^{-(j+1)}l<u\leq 2^{-j}l\rbrace)\\
&\leq  C\mu(B_R)^{-1}\mu(B_{\lambda R})\leq C.
\end{align*}
Therefore,
$$
\frac{\mu(B_{R}\cap\lbrace u\leq 2^{-\hat{j}}l\rbrace)}{\mu(B_R)}\leq C\hat{j}^{-\left(\frac{1-d_2}{d_2}\right)}=C\hat{j}^{-\xi},
$$
where $\xi=\xi(C_D,p,q)>0$.
\end{proof}

\begin{lemma}\label{lemma4frozen} Let $B_{6\lambda R}\Subset \Omega$ be a ball, such that $0<R\leq\min\lbrace 1, \frac{\diam (X)}{12\lambda}\rbrace$. Assume $u,R,\delta$ and $l$ are as in Lemma \ref{lemma3frozen}. If \eqref{eq10frozen} holds, then
\begin{equation}\label{eq14frozen}
\essinf_{B_R} u\geq \frac{l}{C_{1,\delta}},
\end{equation}
for a constant $C_{1,\delta}=C_{1,\delta}(C_{PI},C_D,p,q,\lambda,\delta)$.
\end{lemma}

\begin{proof}
Let $\hat{j}\in\mathbb{N}$ to be determined, and for $i\in\mathbb{N}$, we consider the radii and the levels
$$
r_i=R\left(\frac{1}{2}+\frac{1}{2^{i+1}}\right)\searrow\frac{R}{2},\qquad \tilde{r_i}=\frac{r_{i+1}+r_i}{2},\qquad k_i=\left(\frac{1}{2}+\frac{1}{2^{i+1}}\right)2^{-\hat{j}}l\searrow 2^{-\hat{j}-1}l
$$
By the definition of $DG_{H_0}^{-}(\Omega)$ and \eqref{frozeninduction1}, we have
\begin{align}\label{eq15frozen}
\int_{B_{\tilde{r_i}}}H_0\left(g_{(u-k_i)_-}\right)\dd\mu&\leq C\int_{B_{r_i}}H_0\left(\frac{(u-k_i)_-}{r_i-\tilde{r_i}}\right)\dd\mu\nonumber\\
&\leq C2^{qi}\int_{B_{r_i}}H_0\left(\frac{(u-k_i)_-}{R}\right)\dd\mu\nonumber\\
&\leq C2^{qi}\int_{B_{r_i}}H_0\left(\frac{(u-k_i)_-}{r_i}\right)\dd\mu
\end{align}

Now, we take Lipschitz cut-off functions $\eta_i\in \textrm{Lip}_0(B_{\tilde{r_i}})$, $0\leq \eta_i\leq 1$, such that $\eta_i= 1$ on $B_{\tilde{r_i}}$ and $g_{\eta_i}\leq\frac{2^{i}}{R}$.

On the other hand, by Lemma \ref{lemma3frozen}, and the doubling property of the measure, we get
\begin{align*}
\mu(\lbrace x\in B_{r_{i+1}}: \eta_i(u-k_i)_->0\rbrace)&\leq \mu(\lbrace x\in B_{r_{i+1}}: u\leq 2^{-\hat{j}}l\rbrace)\\
&\leq C\mu(B_R)\hat{j}^{-\xi}\leq C\mu(B_{R/2})\hat{j}^{-\xi}\\
&\leq C\mu(B_{r_{i+1}})\hat{j}^{-\xi}.
\end{align*}
Therefore, we focus on $\hat{j}\in\mathbb{N}$ big enough so that $C\hat{j}^{-\xi}<\frac{1}{2}$. By Theorem \ref{FSPineq0largeSet}, H\"older's inequality, Leibniz' rule, the doubling property of the measure, \eqref{frozeninduction1} and \eqref{eq15frozen}, we deduce

\begin{align*}
\left(\dashint_{B_{r_{i+1}}}H_0\left(\frac{(u-k_i)_-}{r_{i+1}}\right)^{d_1}\dd\mu\right)^{\frac{1}{d1}}&=\left(\dashint_{B_{r_{i+1}}}H_0\left(\frac{\eta_i(u-k_i)_-}{r_{i+1}}\right)^{d_1}\dd\mu\right)^{\frac{1}{d1}}\\
&\leq C\left(\dashint_{B_{\lambda r_{i+1}}}H_0\left(g_{\eta_i(u-k_i)_-}\right)^{d_2}\dd\mu\right)^{\frac{1}{d2}}\\
&\leq C\dashint_{B_{\lambda r_{i+1}}}H_0\left(g_{\eta_i(u-k_i)_-}\right)\dd\mu\\
&\leq C\mu(B_{\lambda r_{i+1}})^{-1}\Bigg(\int_{B_{\tilde{r_i}}}H_0\left(g_{(u-k_i)_-}\right)\dd\mu\\
&\qquad\qquad\qquad\qquad\qquad+\int_{B_{\tilde{r_i}}}H_0\left(\frac{2^{i}(u-k_i)_-}{R}\right)\dd\mu\Bigg)\\
&\leq C\mu(B_{\lambda r_{i+1}})^{-1}2^{qi}\int_{B_{r_i}}H_0\left(\frac{(u-k_i)_-}{r_i}\right)\dd\mu\\
&\leq C2^{qi}\dashint_{B_{r_i}}H_0\left(\frac{(u-k_i)_-}{r_i}\right)\dd\mu.
\end{align*}

Therefore, 
\begin{equation}\label{eq16frozen}
\left(\dashint_{B_{r_{i+1}}}H_0\left(\frac{(u-k_i)_-}{r_{i+1}}\right)^{d_1}\dd\mu\right)^{\frac{1}{d1}}\leq C2^{qi}\dashint_{B_{r_i}}H_0\left(\frac{(u-k_i)_-}{r_i}\right)\dd\mu,
\end{equation}
where $C=C(C_{PI},C_D,q,Q,\lambda,\delta).$
We now estimate from above the right-hand side of \eqref{eq16frozen}. Since $u$ is non-negative, then $(u-k_i)_-\leq\chi_{\lbrace u\leq k_i\rbrace} 2^{-\hat{j}}l$ in $B_R$. On the other hand, in order to estimate from below the left-hand side of \eqref{eq16frozen}, we use

$$
\chi_{\lbrace u\leq k_{i+1}\rbrace} 2^{-\hat{j}-i-2}l=\chi_{\lbrace u\leq k_{i+1}\rbrace}(k_i-k_{i+1})\leq \chi_{\lbrace u\leq k_{i+1}\rbrace}(u-k_i)_-\leq (u-k_i)_.
$$
Therefore, the last three displays and \eqref{frozeninduction2}, yield
\begin{align*}
C2^{-iq}H_0\left(\frac{2^{-\hat{j}}l}{r_{i}}\right)\left(\dashint_{B_{r_{i+1}}}\chi_{\lbrace u\leq k_{i+1}\rbrace}\dd\mu\right)^{\frac{1}{d_1}}&\leq H_0\left(\frac{2^{-\hat{j}-i-2}l}{r_{i}}\right)\left(\dashint_{B_{r_{i+1}}}\chi_{\lbrace u\leq k_{i+1}\rbrace}\dd\mu\right)^{\frac{1}{d_1}}\\
&\leq C2^{iq}H_0\left(\frac{2^{-\hat{j}}l}{r_{i}}\right)\dashint_{B_{r_{i}}}\chi_{\lbrace u\leq k_{i}\rbrace}\dd\mu
\end{align*}
So,
$$
\dashint_{B_{r_{i+1}}}\chi_{\lbrace u\leq k_{i+1}\rbrace}\dd\mu\leq C\left(2^{(q+1)d_1}\right)^i\left(\dashint_{B_{r_{i}}}\chi_{\lbrace u\leq k_{i}\rbrace}\dd\mu\right)^{1+(d_1-1)}.
$$
Therefore, by defining $A_i=\dashint_{B_{r_{i}}}\chi_{\lbrace u\leq k_{i}\rbrace}\dd\mu$, we obtain
$$
A_{i+1}\leq C\left(2^{(q+1)d_1}\right)^iA_i^{1+(d_1-1)},
$$
with $C=C(C_{PI},C_D,p,q,Q,\lambda,\delta)$.
 By Lemma 7.1 in \cite{G}, we have that $A_i\rightarrow 0$ when $i\rightarrow\infty$ provided
 
$$
A_0\leq C^{-\frac{1}{d_1-1}}(2^{(q+1)d_1})^{-\left(\frac{1}{d_1-1}\right)^{2}}=\tilde{C}^{-1}=\tilde{C}^{-1}(C_{PI},C_D,p,q,Q,\lambda,\delta).
$$ 
Meaning, that if 
$$
\frac{\mu(\lbrace x\in B_R: u\leq 2^{-\hat{j}}\rbrace)}{\mu(B_R)}\leq \tilde{C}^{-1},
$$
then, 
$$
\dashint_{B_{R/2}}\chi_{\lbrace u\leq 2^{-\hat{j}-1}\rbrace}\dd\mu=0
$$
Therefore, $\essinf_{B_{R/2}} u\geq\frac{l}{2^{\hat{j}+1}}$ as we want.

By Lemma \ref{lemma3frozen}, we have
$$
\frac{\mu(B_{R}\cap\lbrace u\leq 2^{-\hat{j}}l\rbrace)}{\mu(B_R)}\leq C\hat{j}^{-\xi},
$$
for all $\hat{j}\in\mathbb{N}$ and for some $\xi=\xi(p,q,C_D)>0$ and $C=C(C_{PI},C_D,q,\lambda,\delta)$.

By asking $\hat{j}$ to be large enough so that $C\hat{j}^{-\xi}\leq\min\lbrace\frac{1}{2},\tilde{C}^{-1}\rbrace$, we obtain
$$
\essinf_{B_{R/2}} u\geq\frac{l}{2^{\hat{j}+1}}=\frac{l}{C_{1,\delta}},
$$

where $C_{1,\delta}=C_{1,\delta}(C_{PI},C_D,p,q,\lambda,\delta)$.

\end{proof}
Lastly, the following weak-Harnack inequality for frozen functionals follows from Lemma \ref{lemma4frozen} exactly as in proof of Theorem 7.1 of \cite{KS}.

\begin{theorem}\label{essinfharnackfrozen}Let $B_{6\lambda R}\Subset \Omega$ be a ball, such that $0<R\leq\min\lbrace 1, \frac{\diam (X)}{12\lambda}\rbrace$. Assume $u\in DG_{H_0}^{-}(\Omega)$ is non negative. Then, there exist a positive exponent $\sigma$ and a positive constant $C$ both independent of the ball $B_R$ and of $a_0$, such that

$$
\inf_{B_{3R}} u\geq C\left(\dashint_{B_R}u^{\sigma}\dd\mu\right)^{\frac{1}{\sigma}}.
$$
\end{theorem}
\begin{proof}
As mentioned before, it follows by Lemma \ref{lemma4frozen} as in Theorem 7.1 of \cite{KS}.
\end{proof}

\section{H\"{o}lder continuity}\label{Sec4}

This section aims to prove H\"older continuity of a local quasiminimizer $u$, see Theorem \ref{localholder} below. Recall that in Section \ref{Boundedness} we proved that if $u$ is a local quasiminimizer, then $u\in L^{\infty}_{\loc}(\Omega).$

From now on, we assume that $\mu$ is upper $Q$-Ahlfors regular, that is there exists a constant $C_1>0$ such that our measure satisfies the following inequality
\begin{equation}\label{upper Q-Ahlfors}
\mu(B_r(x))\leq C_1 r^Q \quad \mbox{for every $x\in X$ and $0<r\leq {\rm diam(X)}.$}
\end{equation}
This assumption ensures some uniformity and regularity in the distribution of the measure. We note that this a general assumption. For example, self-similar fractals, metric measure spaces with controlled curvature, uniformly rectifiable sets and Carnot groups can exhibit upper Ahlfors regularity.

We start by proving some auxiliary results. The key one is the following almost standard Caccioppoli's inequality, which states that, in the $p$-regime case, double phase local quasiminima satisfy a Caccioppoli type inequality analogous to the one satisfied for quasiminima of functionals with just $p$-growth plus some extra controllable terms. In the euclidean case this result has been proven by Colombo-Mingione \cite{CM}.
\begin{lemma}[Almost standard Caccioppoli's inequality]\label{Almost standard Caccioppoli's inequality}
  Assume that  $u\in N^{1,1}_{\loc}({\Omega})$ with $H(\cdot,g_u)\in L^1_{\loc}(\Omega)$  is a local quasiminimizer in $\Omega$ such that 
  \begin{equation}\label{p-regime}
      \sup_{B_R} a(x) \leq C [a]_{\alpha} \mu(B_R)^{\frac{\alpha}{Q}}.
  \end{equation}
  Then there exists $C= C(C_1, p, q, [a]_{\alpha}, \|u\|_{L^{\infty}(B_R)}, K)>0$ such that for any choice of concentric balls $B_s\subset B_t \subset B_R\Subset\Omega $, with $0<t<s\leq R\leq 1$, the following inequality 
\begin{align}\label{ASCaccioppoli}
\int_{B_t}g_{(u-k)_+}^p\, \dd \mu	& \leq C \left(\left(\frac{R}{s-t}\right)^q\int_{B_s} \left|\frac{(u-k)_+}{R}\right|^p\, \dd\mu\right)
\end{align}
is satisfied, where $(u-k)_+= \max\{u-k,0\}$.   
\end{lemma}
\begin{proof}
 By Lemma \ref{DeGiorgiLemma}, since $u$ is a local quasiminimizer, then it satisfies the double phase Caccioppoli's inequality \eqref{5.4}. Therefore, we have 
 \begin{align}\label{asc1}
\int_{B_t}g_{(u-k)_+}^p\, \dd \mu &\leq   \int_{B_t}H(x,g_{(u-k)_+})\, \dd \mu    \leq  C \int_{B_s} H\left(x,\frac{(u-k)_+}{s-t}\right)\, \dd\mu \nonumber\\
& = C \left(\int_{B_s} \left|\frac{(u-k)_+}{s-t}\right|^p\, \dd\mu + \int_{B_s} a(x)\left|\frac{(u-k)_+}{s-t}\right|^q\, \dd\mu\right)\nonumber\\
& =  C \left(\left(\frac{R}{s-t}\right)^p\int_{B_s} \left|\frac{(u-k)_+}{R}\right|^p\, \dd\mu +\left(\frac{R}{s-t}\right)^q \int_{B_s} a(x)\left|\frac{(u-k)_+}{R}\right|^q\, \dd\mu\right).
 \end{align}
Now, we estimate the integrand function in the last term of the previous inequality by using \eqref{p-regime} and  \eqref{upper Q-Ahlfors}, we get
\begin{align*}
a(x)\left|\frac{(u-k)_+}{R}\right|^q &\leq C[a]_{\alpha}\frac{\mu(B_R)^{\frac{\alpha}{Q}}}{R^q} \|u\|^{q-p}_{L^{\infty}(B_s)} |(u-k)_+|^p\\
& \leq C[a]_{\alpha}\frac{C_1 R^{\alpha}}{R^q} \|u\|^{q-p}_{L^{\infty}(B_R)} |(u-k)_+|^p\\
&= \frac{C[a]_{\alpha}}{R^{q-\alpha}} \|u\|^{q-p}_{L^{\infty}(B_R)} |(u-k)_+|^p\\
&\leq \frac{C [a]_{\alpha}}{R^{p}}\|u\|^{q-p}_{L^{\infty}(B_R)} |(u-k)_+|^p= C \left|\frac{(u-k)_+}{R}\right|^p,
\end{align*}
 where $C=C(C_1, p, q, [a]_{\alpha}, \|u\|_{L^{\infty}(B_R)})$.
 Thus, \eqref{asc1} becomes
\begin{align*}
\int_{B_t}g_{(u-k)_+}^p\, \dd \mu &\leq  C \left(\left(\frac{R}{s-t}\right)^p\int_{B_s} \left|\frac{(u-k)_+}{R}\right|^p\, \dd\mu +C\left(\frac{R}{s-t}\right)^q\int_{B_s} \left|\frac{(u-k)_+}{R}\right|^p\, \dd\mu\right)\\
& \leq C \left(\frac{R}{s-t}\right)^q \int_{B_s} \left|\frac{(u-k)_+}{R}\right|^p\, \dd\mu.
 \end{align*}
\end{proof}

\begin{lemma}\label{localLemma1} Let $B_{6\lambda R}\Subset \Omega$ be a ball, such that $0<R\leq\min\lbrace 1, \frac{\diam (X)}{12\lambda}\rbrace$ such that \eqref{p-regime} holds for the concentric ball $B_{3\lambda R}\Subset\Omega$. 
Assume $u\in N^{1,1}_{\textrm{loc}}(\Omega)$, with $H(\cdot, g_u)\in L^1_{\textrm{loc}}(\Omega)$ belongs to $DG_H(\Omega)$. Let $h>0$ such that the density condition
    \begin{equation}\label{local1}
    \mu(S_{h,R})\leq\gamma\mu(B_R),
    \end{equation}
    holds for some $\gamma\in(0,1).$ Then, there exist a positive constant $C=C(\data, C_1, \|u\|_{L^{\infty}(B_R)}, \gamma)$ such that for any $k<h$ the following inequality
    \begin{align}\label{local2}
    (k-h)\mu(S_{k,R})&\leq C \mu(B_R)^{1-\frac{1}{s}}\left(\mu(S_{h,\lambda R})-\mu(S_{k,\lambda R})\right)^{\frac{1}{s}-\frac{1}{p}}\left(\int_{S_{h,2\lambda R}}
(u-h)^p\dd\mu\right)^{\frac{1}{p}}
\end{align}
holds.
\end{lemma}

\begin{proof} 
For $k> h$, let us define
\begin{equation}\label{v(x)}
	v = \min\{u, k\} - \min\{u, h\}.
\end{equation} Note that, by \eqref{v(x)}, we have
\begin{equation} \label{4.0bisineq}
	v=\begin{cases}
		0 \hspace{1.98cm} \mbox{if $u\leq h<k$,}\\
		u-h \hspace{1.3cm} \mbox{if $h<u<k$,}\\
		k-h \hspace{1.3cm} \mbox{if $u\geq k>h$.}
	\end{cases}
\end{equation}

From \eqref{local1} and \eqref{4.0bisineq}, we deduce that $\mu(\{x \in B_R: v(x) > 0\}) \leq \gamma\mu(B_R).$ We recall that the space supports a weak $(1,s)$-Poincar\'e inequality for some $1<s<p<q$ (see Theorem \ref{kz}). By using H\"{o}lder inequality  and Lemma 2.1 of \cite{KS} with $t=s$, we obtain
\begin{align*}
    (k-h)\mu(S_{k,R})&=\int_{S_{k,R}}v\,\dd \mu\leq \int_{B_R}v\,\dd\mu\leq \mu(B_R)^{1-\frac{1}{s}}\left(\int_{B_R}v^s\dd\mu\right)^{\frac{1}{s}}\\
&\leq C R\,\mu(B_R)^{1-\frac{1}{s}}\left(\int_{B_{\lambda R}}
g_{v}^s \dd\mu\right)^{\frac{1}{s}}\\
& \leq CR\,\mu(B_R)^{1-\frac{1}{s}}\left(\int_{S_{h,\lambda R}\setminus S_{k,\lambda R}}
g_{v}^s \dd\mu\right)^{\frac{1}{s}}\\
&\leq CR\,\mu(B_R)^{1-\frac{1}{s}}\left(\int_{S_{h,\lambda R}}
g_{v}^p \dd\mu\right)^{\frac{1}{p}} \left(\mu(S_{h,\lambda R})-\mu(S_{k,\lambda R})\right)^{\frac{1}{s}-\frac{1}{p}}.
\end{align*}
By the almost standard Caccioppoli's inequality (Lemma \ref{Almost standard Caccioppoli's inequality}), we get
\begin{align*}
    (k-h)\mu(S_{k,R})&\leq CR\mu(B_R)^{1-\frac{1}{s}}\left(C\left(\frac{\lambda R}{\lambda R}\right)^q\int_{S_{h,2\lambda R}}
\left|\frac{u-h}{\lambda R}\right|^p\dd\mu\right)^{\frac{1}{p}}\left(\mu(S_{h,\lambda R})-\mu(S_{k,\lambda R})\right)^{\frac{1}{s}-\frac{1}{p}}\\
&\leq C\mu(B_R)^{1-\frac{1}{s}}\left(\int_{S_{h,2\lambda R}}
(u-h)^p\dd\mu\right)^{\frac{1}{p}}\left(\mu(S_{h,\lambda R})-\mu(S_{k,\lambda R})\right)^{\frac{1}{s}-\frac{1}{p}},
\end{align*}
where $C=C(\data, C_1, \|u\|_{L^{\infty}(B_R)}, \gamma)$.
\end{proof}

\noindent For any $B_\rho \Subset \Omega$, we shall denote
$$m(\rho) = \essinf_{B_\rho}u,\ \ M(\rho) = \esssup_{B_\rho}u$$
and
$$
\osc(u,\rho)=M(\rho)-m(\rho).
$$

\begin{lemma}\label{localLemma2} 
    Let $B_{6\lambda R}\Subset\Omega$ be a ball, with $0<R\leq\min\lbrace 1, \frac{\diam (X)}{12\lambda}\rbrace$, such that \eqref{p-regime} holds for the concentric ball $B_{3\lambda R}\Subset\Omega$. 
	Let $M =M(3\lambda R)$, $m = m(3 \lambda R)$.
Let $u\in N^{1,1}_{\textrm{loc}}(\Omega)$ with $H(\cdot, g_u)\in L^1_{\textrm{loc}}(\Omega)$ belong to $DG_H(\Omega)$. Assume $u$ is bounded from below. Moreover, assume that the density condition
 \begin{equation}\label{estimateSB}
     \mu(S_{k_0,R}) \leq \gamma \mu(B_R),\quad\textrm{for some }0 < \gamma < 1,
 \end{equation}
 holds for $k_0= \frac{M+m}{2}$. Then there exists a positive constant $C=C(\data,s, C_1, \gamma, \|u\|_{L^{\infty}(B_R)})$ such that
 \begin{equation}\label{quotientn}
    \dfrac{\mu(S_{k_n}, R)}{\mu(B_R)} \leq C n^{-\zeta},
 \end{equation} with $\zeta=\zeta(p,s)>0$, $n$ positive integer and $k_n=M-2^{-(n+1)} \osc(u,3\lambda R)$.
\end{lemma}
\begin{proof} 
 Let $k_j=M-2^{-(j+1)} (M-m)$, $j\in \mathbb{N} \cup \{0\}$. Therefore, $\lim_{j \to +\infty}k_j = M$. Note that $M-k_{j-1}=2^{-j} (M-m)$ and $k_j-k_{j-1}=2^{-(j+1)} (M-m)$. By \eqref{local2} in Lemma \ref{localLemma1} and the doubling property, we deduce
	\begin{align*}
	2^{-(j+1)} (M-m)\mu(S_{k_j,R})&	=(k_j-k_{j-1}) \mu(S_{k_j,R})\\ &\leq C\mu(B_R)^{1-\frac{1}{s}}\left(\mu(S_{k_{j-1},\lambda R})-\mu(S_{k_{j},\lambda R})\right)^{\frac{1}{s}-\frac{1}{p}}\left(\int_{S_{k_{j-1},2 \lambda R}}
(u-k_{j-1})^p\dd\mu\right)^{\frac{1}{p}}\\
&\leq C\mu(B_R)^{1-\frac{1}{s}+\frac{1}{p}}\left(\mu(S_{k_{j-1},\lambda R})-\mu(S_{k_{j},\lambda R})\right)^{\frac{1}{s}-\frac{1}{p}}(M-k_{j-1})\\
&= C\mu(B_R)^{1-\frac{1}{s}+\frac{1}{p}}\left(\mu(S_{k_{j-1},\lambda R})-\mu(S_{k_{j},\lambda R})\right)^{\frac{1}{s}-\frac{1}{p}}2^{-j} (M-m).
	\end{align*}
where $C=C(\data,C_1,\gamma,\|u\|_{L^{\infty}(B_R)})$.
Therefore, for every $j\geq 0$, we obtain
	\begin{align*}
	\mu(S_{k_j,R}) \leq C \mu(B_R)^{1-\frac{1}{s}+\frac{1}{p}} \Big(\mu(S_{k_{j-1}, \lambda R})&-\mu(S_{k_j, \lambda R})\Big)^{\frac{1}{s}-\frac{1}{p}}.
	\end{align*}
 
\noindent If $n>j $, then $\mu(S_{k_n,  R})\leq \mu(S_{k_j,  R})$, and so
	\begin{align*}
		\mu(S_{k_n,  R})\leq  C \mu(B_R)^{1-\frac{1}{s}+\frac{1}{p}} \Big(\mu(S_{k_{j-1}, \lambda R})-\mu(S_{k_j, 2
  \lambda R})\Big)^{\frac{1}{s}-\frac{1}{p}}.
	\end{align*}
	By summing the above inequality over $j=0,..., n-1$ and using \eqref{estimateSB}, we get
	\begin{align*}
		n\mu(S_{k_n,  R})^{\frac{sp}{p-s}}&\leq C \mu(B_R)^{\frac{sp-p+s}{p-s}}\,\Big(\mu(S_{k_{0},\lambda R})-\mu(S_{k_n, \lambda R})\Big) \\
&\leq C \mu(B_R)^{\frac{sp-p+s}{p-s}}\,\mu(S_{k_{0}, \lambda R}) \\
&\leq C \mu(B_R)^{\frac{sp-p+s}{p-s}}\,\mu(S_{k_{0}, R}) \\
  &\leq C  \mu(B_R)^{\frac{sp}{p-s}}.
	\end{align*}
 Therefore, 
 \begin{equation*}
    \dfrac{\mu(S_{k_n}, R)}{\mu(B_R)}\leq C n^{-\frac{p-s}{sp}} \leq C n^{-\zeta},
 \end{equation*} with $\zeta=\zeta(p,s)>0$ and $C=C(\data,s, C_1, \gamma, \|u\|_{L^{\infty}(B_R)})$. 
 
 In particular, we have $\lim_{n\to+\infty}\mu(S_{k_n  R})=0$ and, since 
	$\mu(S_{k,  R})$ is a monotonic decreasing function of $k$, we conclude that $\lim_{k \to M}\mu(S_{k,R})=0$.
\end{proof}

\begin{remark}\label{-uinsteadofu}
	We note that it is not restrictive to suppose that 
	\begin{equation}\label{muSalpha}
		\mu(S_{k, R})\leq \frac{\mu(B_R)}{2} \quad\mbox{for all $k \in \mathbb{R}$, $R>0$ with $B_R \subset \Omega$.}
	\end{equation}
	In fact, if $$\mu(S_{k, R})=\mu(\{x \in B_R:u(x)>k\})> \frac{\mu(B_R)}{2},$$ then
	$$\mu(\{x \in B_R:-u(x)\leq - k\})>\frac{\mu(B_R)}{2},$$
	and so, $$\mu(\{x \in B_R:-u(x)>- k\})<\frac{\mu(B_R)}{2}.$$
	That is, inequality \eqref{muSalpha} holds true considering $-u$ instead of $u$.
\end{remark}

\begin{lemma}\label{beforeosc} Let $B_{6\lambda R}\Subset\Omega$ be a ball, with $0<R\leq\min\lbrace 1, \frac{\diam (X)}{12\lambda}\rbrace$, such that \eqref{p-regime} holds for the ball $B_{3\lambda R}\Subset\Omega$. 
	Assume $u\in N^{1,1}_{\textrm{loc}}(\Omega)$ with $H(\cdot, g_u)\in L^1_{\textrm{loc}}(\Omega)$ belongs to $DG_H(\Omega)$ and it is bounded from below. Then, for every $0\leq r\leq 3\lambda R$ and every $\kappa\in (0,1)$, there exists $\sigma\in (0,1)$, $\sigma=\sigma(\data, C_1,\Vert u\Vert_{L^{\infty}(B_r)}, \kappa)$, such that, if for some $\varepsilon>0$ the density condition
 \begin{equation}\label{densitycondMosc}
     \mu\left(S_{M(r)-\varepsilon\osc(u,r), r}\right)\leq \sigma\mu(B_r)
 \end{equation}
holds, then
\begin{equation}\label{conclusiondensity}
    u(x)\leq M(r)-\kappa\varepsilon\osc(u,r)
\end{equation}
holds a.e. in $B_{r/2}$.
 \end{lemma}

\begin{proof}
    Consider a sequence of nested balls $\lbrace B_{\rho_i}\rbrace$ concentric to $B_R$ for $i\geq 0$, where $\rho_i=\frac{r}{2}(1+2^{-i})\searrow\frac{r}{2}$, and define also $\bar{\rho_i}=(\rho_i+\rho_{i+1})/2$. We use Lipschitz cut-off functions $\eta_i\in \textrm{Lip}_0(B_{\bar{\rho_i}})$ such that $\eta_i= 1$ on $B_{\rho_i}$ and $g_{\eta_i}\leq\frac{C2^{i}}{r}$. We also define the levels $k_i=M(r)-\kappa\varepsilon\osc(u,r)-(1-\kappa)\varepsilon\osc(u,r)/2^{i}\nearrow M(r)-\kappa\varepsilon\osc(u,r)$. Lastly, we define $v_i=\eta_i(u-k_i)_+$. By definition and Leibniz's rule, we obtain

$$
g_{v_i}\leq g_{(u-k_i)_+}\eta_i+(u-k_i)_+ g_{\eta_i}\leq g_{(u-k_i)_+}\eta_i+\frac{C2^i}{r}(u-k_i)_+\qquad\textrm{in }B_{\bar{\rho_i}}.
$$

\noindent Furthermore, for a $\sigma\in (0,1)$, that is yet to be determined, we have

$$
\mu(\lbrace x\in B_r:v(x)>0\rbrace)\leq \mu(S_{k_i,r})\leq \mu(S_{k_0,r})\leq \sigma \mu(B_r).
$$
So that, by Lemma 2.1 of \cite{KS}, H\"older's inequality and Lemma \ref{Almost standard Caccioppoli's inequality}, we obtain
\begin{align*}
    \left(\frac{(1-\kappa)\varepsilon\osc(u,r)}{2^{i+1}}\right)^{p}\mu(S_{k_{i+1},\rho_{i+1}})&=(k_{i+1}-k_i)^{p}\mu(S_{k_{i+1},\rho_{i+1}})\\
    &\leq C\int_{S_{k_{i},\bar{\rho_{i}}}}v_i^p\dd\mu\leq\mu(S_{k_{i},\bar{\rho_{i}}})^{\frac{p}{Q}}\left(\int_{S_{k_{i},\bar{\rho_{i}}}}v_i^{p^*}\dd\mu \right)^{\frac{p}{p^*}}\\
    &\leq C \mu(S_{k_{i},\bar{\rho_{i}}})^{\frac{p}{Q}}\left(\dashint_{B_{r}}\vert v_i\vert^{p^*}\dd\mu\right)^{\frac{p}{p^*}}\mu(B_r)^{\frac{p}{p^*}}\\
    &\leq C r^p\mu(S_{k_{i},\bar{\rho_{i}}})^{\frac{p}{Q}}\left(\int_{B_{\lambda r}}g_{v_i}^{p}\dd\mu\right)\mu(B_r)^{-\frac{p}{Q}}\\
    &\leq C r^p\mu(S_{k_{i},\bar{\rho_{i}}})^{\frac{p}{Q}}\mu(B_{\rho_i})^{-\frac{p}{Q}}\left(\int_{B_{\bar{\rho_i}}}g_{(u-k_i)_+}^{p}\dd\mu+\frac{C2^{ip}}{r^p}\int_{B_{\bar{\rho_i}}}(u-k_i)_+^p\dd\mu\right)\\
     &\leq C r^p\mu(S_{k_{i},\bar{\rho_{i}}})^{\frac{p}{Q}}\mu(B_{\rho_i})^{-\frac{p}{Q}}\Bigg(\left(\frac{R}{\rho_i-\bar{\rho_i}}\right)^q\int_{B_{\rho_i}}\left\vert\frac{u-k_i}{r}\right\vert^{p}\dd\mu\\
&\qquad\qquad\qquad\qquad\qquad\qquad\qquad\qquad+C2^{ip}\int_{S_{k_i,\bar{\rho_i}}}\left\vert\frac{(u-k_i)_+}{r}\right\vert^{p}\dd\mu\Bigg)\\
&\leq C\mu(S_{k_{i},\bar{\rho_{i}}})^{\frac{p}{Q}}\mu(B_{\rho_i})^{-\frac{p}{Q}}2^{iq}(M(r)-k_i)^p\mu(S_{k_i,\rho_i})\\
&\leq C\mu(S_{k_{i},\rho_{i}})^{1+\frac{p}{Q}}\mu(B_{\rho_i})^{-\frac{p}{Q}}2^{iq}(\varepsilon\osc(u,r))^p.
\end{align*}
Therefore, 
$$
 \left(\frac{(1-\kappa)\varepsilon\osc(u,r)}{2^{i+1}}\right)^{p}\mu(S_{k_{i+1},\rho_{i+1}})\leq C\mu(S_{k_{i},\rho_{i}})^{1+\frac{p}{Q}}\mu(B_{\rho_i})^{-\frac{p}{Q}}2^{iq}(\varepsilon\osc(u,r))^p.
$$
Dividing both sides of the previous inequality by $\mu(B_{\rho_i})$, we obtain
$$
\frac{\mu(S_{k_{i+1},\rho_{i+1}})}{\mu(B_{\rho_{i+1}})}\leq \frac{\mu(S_{k_{i+1},\rho_{i+1}})}{\mu(B_{\rho_{i}})}\leq C \left(\frac{\mu(S_{k_{i},\rho_{i}})}{\mu(B_{\rho_{i}})}\right)^{1+\frac{p}{Q}}(4^q)^i(1-\kappa)^p,
$$
where $C=C(\data,C_1, \Vert u\Vert_{L^{\infty}(B_r)},\sigma)$. Therefore, if we define $\Sigma_i=\frac{\mu(S_{k_{i},\rho_{i}})}{\mu(B_{\rho_{i}})}$, we obtain the recursive estimate
$$
\Sigma_{i+1}\leq \frac{C(4^q)^{i}}{(1-\kappa)^p}\Sigma_{i}^{1+\frac{p}{Q}}, \qquad C>0.
$$
In order to prove \eqref{conclusiondensity}, we need $\Sigma_i\rightarrow 0$ as $i\rightarrow +\infty$. By Lemma 7.1 in \cite{G}, this happens if 
$$
\Sigma_0\leq\left(\frac{C}{(1-\kappa)^p}\right)^{-\frac{Q}{p}}(4^q)^{-\frac{Q^2}{p}}.
$$
Meaning, 
$$
\frac{\mu(S_{M(r)-\varepsilon\osc(u,r),r})}{\mu(B_r)}<\sigma\leq \left(\frac{C}{(1-\kappa)^p}\right)^{-\frac{Q}{p}}(4^q)^{-\frac{Q^2}{p}}.
$$
Therefore, by choosing $\sigma=\sigma(\data,C_1, \Vert u\Vert_{L^{\infty}(B_r)},\kappa)$, small enough, we indeed obtain \eqref{conclusiondensity} as wanted.

\end{proof}

The following lemma is an oscillation reduction result. 

\begin{lemma}\label{localLemma3} Let $B_{6\lambda R}\Subset\Omega$ be a ball, with $0<R\leq\min\lbrace 1, \frac{\diam (X)}{12\lambda}\rbrace$, such that \eqref{p-regime} holds for the ball $B_{3\lambda R}\Subset\Omega$. 
	Assume $u\in N^{1,1}_{\textrm{loc}}(\Omega)$ with $H(\cdot, g_u)\in L^1_{\textrm{loc}}(\Omega)$ belongs to $DG_H(\Omega)$ and it is bounded from below. Let $0<\rho<R$. Then, there exists $0 < \eta < 1$ such that
	$$\osc(u, 3\lambda \rho) \leq 4^\eta
	\left(\frac{\rho }{R}\right)^\eta \osc(u, 3\lambda R),$$
 with $\eta=\eta(\data,C_1,\Vert u\Vert_{L^{\infty}(B_{3\lambda R})})$.
\end{lemma}

\begin{proof}
The proof follows exactly as in Proof of Lemma \ref{lemma8frozen}.
\end{proof}

Now, we are ready to prove the main theorem of this section that is a H\"older continuity result for local quasiminima.
The proof is obtained by merging the estimates obtained both for the frozen functional and the $(p,q)$-phase.

\begin{theorem}\label{localholder} Let $u\in N^{1,1}_{\textrm{loc}}(\Omega)$ with $H(\cdot, g_u)\in L^1_{\textrm{loc}}(\Omega)$ be a local quasiminimizer. Then $u$ is locally H\"older continuous.
\end{theorem}
\begin{proof}
Since $u$ is a local quasiminimizer, then by Lemma \ref{DeGiorgiLemma}, $u$ satisfies the double phase Caccioppoli inequality \eqref{5.4}. Let $\Omega'\Subset\Omega''\Subset \Omega$, and $0<R_0<\min\lbrace 1, \frac{\diam(X)}{12\lambda}\rbrace$ such that $B_{6\lambda R_0}\subset\Omega''$. For $k\in\mathbb{N}$, we also consider the following conditions
\begin{equation}\label{local14}
\sup_{x\in B_{\frac{3\lambda R_0}{4^k}}}a(x)\leq C[a]_{\alpha}\mu(B_{\frac{3\lambda R_0}{4^k}})^{\alpha/Q},    
\end{equation}
with $C>0$. Let us define the index
\begin{equation*}\label{local15}
    m=\min\lbrace k\in\mathbb{N}\cup\lbrace \infty\rbrace:\ \textrm{\eqref{local14} fails}\rbrace.
\end{equation*}
We can apply Lemma \ref{localLemma3}, with $\rho=\frac{R_0}{4^k}$, and $R=R_0$, to obtain
\begin{equation}\label{local16}
\osc\left(u,\frac{3\lambda R_0}{4^k}\right)\leq 4^{\eta}\theta^k\osc(u,3\lambda R_0),    
\end{equation}
for all $k\geq 1$, where $\eta\in(0,1)$, $\eta=\eta(\data, C_1,\Vert u\Vert_{L^{\infty}(\Omega'')})$ and $\theta=4^{-\eta}<1$. In the case $m<\infty$, then condition \eqref{local14} fails at the ball $B_{\frac{3\lambda R_0}{4^m}}$. Therefore, there exists a point $x_0\in B_{\frac{3\lambda R_0}{4^m}}$ such that
\begin{equation}\label{local17}
    a(x_0)>C[a]_{\alpha}\mu(B_{\frac{3\lambda R_0}{4^m}})^{\alpha/Q}.
\end{equation}
We now want to prove that $u$ is a local quasiminimizer on $B_{\frac{3\lambda R_0}{4^m}}$ of the frozen functional $$H_0(z)=\vert z\vert^p+a(x_0)\vert z\vert^q.$$ Meaning, we want to prove that there exists a constant $\tilde{K}\geq 1$, such that
\begin{equation}\label{local18}
\int_{\omega\cap\lbrace u\neq v\rbrace}H_0(g_u)\dd\mu\leq \tilde{K}\int_{\omega\cap\lbrace u\neq v\rbrace}H_0(g_v)\dd\mu,    
\end{equation}
for every open subset $\omega\Subset B_{\frac{3\lambda R_0}{4^m}}$, and for every function $v\in N^{1,1}(\omega)$ with $u-v\in N_0^{1,1}(\omega)$.

\noindent Since $u$ is a local quasiminimizer, there exists $K\geq 1$, such that
\begin{equation}\label{local19}
\int_{\omega\cap\lbrace u\neq v\rbrace}H(x, g_u)\dd\mu\leq K\int_{\omega\cap\lbrace u\neq v\rbrace}H(x, g_v)\dd\mu,    
\end{equation}
for every open subset $\omega\Subset \Omega$, and for every function $v\in N^{1,1}(\omega)$ with $u-v\in N_0^{1,1}(\omega)$. In particular, since $B_{\frac{3\lambda R_0}{4^m}}\Subset \Omega$, inequality \eqref{local19} holds for every open subset $\omega\Subset B_{\frac{3\lambda R_0}{4^m}}$. Note that, by \eqref{local17}, for any $x\in B_{\frac{3\lambda R_0}{4^m}}$ we have

\begin{align*}
    2a(x)&=2a(x_0)-2(a(x_0)-a(x))\geq a(x_0)+C[a]_{\alpha}\mu(B_{\frac{3\lambda R_0}{4^m}})^{\alpha/Q}-2\vert a(x_0)-a(x)\vert\\
    &\geq a(x_0)+C\sup_{x,y \in B_{\frac{3\lambda R_0}{4^m}},\ x\neq y} \dfrac{|a(x)-a(y)|}{\delta_{\mu}(x,y)^{\alpha}}\mu(B_{\frac{3\lambda R_0}{4^m}})^{\alpha/Q}-2\vert a(x_0)-a(x)\vert\\
    &\geq a(x_0)+2\vert a(x_0)-a(x)\vert-2\vert a(x_0)-a(x)\vert\\
    &=a(x_0).
\end{align*}
Therefore, $2a(x)\geq a(x_0)$ for all $x\in B_{\frac{3\lambda R_0}{4^m}}$. On the other hand,
\begin{align*}
    2a(x_0)&=a(x_0)+a(x_0)\geq a(x_0)+C[a]_{\alpha}\mu(B_{\frac{3\lambda R_0}{4^m}})^{\alpha/Q}\\
    &=a(x)-(a(x)-a(x_0))+C[a]_{\alpha}\mu(B_{\frac{3\lambda R_0}{4^m}})^{\alpha/Q}\\
    &\geq a(x)-\vert a(x)-a(x_0)\vert+ C\sup_{x,y \in B_{\frac{3\lambda R_0}{4^m}},\ x\neq y} \dfrac{|a(x)-a(y)|}{\delta_{\mu}(x,y)^{\alpha}}\mu(B_{\frac{3\lambda R_0}{4^m}})^{\alpha/Q}\\
    &\geq a(x)-\vert a(x)-a(x_0)\vert+\vert a(x)-a(x_0)\vert\\
    &= a(x).
\end{align*}
Therefore, $4a(x)\geq 2a(x_0)\geq a(x)$, for all $x\in B_{\frac{3\lambda R_0}{4^m}}$. So, by \eqref{local18}, the previous inequality, and \eqref{local19}, for $\omega\Subset B_{\frac{3\lambda R_0}{4^m}}$, we have 

\begin{align*}
    \int_{\omega\cap\lbrace u\neq v\rbrace}H_0(g_u)\dd\mu&=\int_{\omega\cap\lbrace u\neq v\rbrace}(g_u^p+a(x_0)g_u^q)\dd\mu\\
    &\leq\int_{\omega\cap\lbrace u\neq v\rbrace}(g_u^p+2a(x)g_u^q)\dd\mu\\
    &\leq2\int_{\omega\cap\lbrace u\neq v\rbrace}(g_u^p+a(x)g_u^q)\dd\mu =2\int_{\omega\cap\lbrace u\neq v\rbrace}H(x,g_u)\dd\mu\\
    &\leq 2K \int_{\omega\cap\lbrace u\neq v\rbrace}H(x,g_v)\dd\mu=2K\int_{\omega\cap\lbrace u\neq v\rbrace}(g_v^p+a(x)g_v^q)\dd\mu\\
    &\leq 2K\int_{\omega\cap\lbrace u\neq v\rbrace}(g_v^p+2a(x_0)g_v^q)\dd\mu\\
    &\leq 4K\int_{\omega\cap\lbrace u\neq v\rbrace}(g_v^p+a(x_0)g_v^q)\dd\mu=4K\int_{\omega\cap\lbrace u\neq v\rbrace}H_0(g_v)\dd\mu.    
\end{align*}
Hence, \eqref{local18} holds. Then, $u$ is a local quasiminimizer of the frozen functional $H_0$ in $B_{\frac{3\lambda R_0}{4^m}}$, with $\tilde{K}=4K$. 

\noindent Therefore, we can now apply Lemma \ref{lemma8frozen}, with $\rho=\tau^h_0\frac{3R_0}{4^m}$, where $0<\tau_0<\frac{1}{8}$, and $R=\frac{3R_0}{2\cdot 4^m}$, to obtain
\begin{align*}
    \osc\left(u, \tau^h_0\frac{3\lambda R_0}{4^m}\right)&\leq 4^{\eta}(2\tau_0^h)^{\eta}\osc\left(u, \frac{3\lambda R_0}{4^m}\right)\\
    &\leq 4^{\eta}\theta^h\osc\left(u, \frac{3\lambda R_0}{4^m}\right),
\end{align*}
where $0<\theta=4^{-\eta}<1$. We note that $\eta$ does not depend on $a(x_0)$.

\noindent Thus, for all $h\geq 1$
$$
\osc\left(u, \tau^h_0\frac{3\lambda R_0}{4^m}\right)\leq 4^{\eta}\theta^h\osc\left(u, \frac{3\lambda R_0}{4^m}\right).
$$
Using this last inequality, together with \eqref{local16}, it yields
\begin{equation}\label{local20}
\osc\left(u, \tau^h_03\lambda R_0\right)\leq 4^{\eta}\theta^h\osc\left(u, 3\lambda R_0\right),  
\end{equation}
for all $h\geq 1$, $\theta\in(0,1).$ Note that, this last inequality also holds in the case $m=\infty$, directly with $\tau_0=\frac{1}{8}$, by \eqref{local18}. Then H\"older continuity of $u$ in $\Omega'$ follows from inequality \eqref{local20} and a standard covering argument since all the previous considerations are independent of the starting ball $B_{R_0}$, as long as  $B_{6\lambda R_0}$ is contained in $\Omega'$.
\end{proof}

\section{Harnack inequality}\label{Sec5}
This section aims to prove Harnack inequality for local quasiminima. 
\begin{theorem}[Harnack inequality]\label{Corollary 7.3 KS}
Let $u\in N^{1,1}_{\textrm{loc}}(\Omega)$ with $H(\cdot, g_u)\in L^1_{\textrm{loc}}(\Omega)$ be a non-negative local quasiminimizer. 
Then, for every ball $B_R$ with $0<R<\min\lbrace 1,\frac{\diam(X)}{18\lambda}\rbrace$, there exists a constant $C$, such that

	$$\esssup_{B_R} u \leq C\,\essinf_{B_R} u$$
	holds. 
 We note that $C$ does not dependent on the radius $R$.
\end{theorem}
To prove Theorem \ref{Corollary 7.3 KS}, we need to premise various lemmata. We suppose a bound from above for the coefficient function $a(\cdot)$. 
We fix a ball $B_{9\lambda R}\Subset \Omega$, with $0<R<\min\lbrace 1, \frac{\diam(X)}{18\lambda}\rbrace$.  
  We assume that the following condition holds on the concentric ball $B_{6\lambda R}$ \begin{equation}\label{eq1harnack}
   \sup_{x\in B_{6\lambda R}}a(x)\leq C[a]_{\alpha}\mu(B_{6\lambda R})^{\alpha/Q}, 
\end{equation}
holds, with $C>0$.
Basically, we are now working in the $p$-phase, in the sense that 
estimates return to a form that closely resemble the ones associated with quasiminima of the $p$-Dirichlet integral.

\begin{lemma} Assume that \eqref{eq1harnack} holds for the concentric ball $B_{3\lambda R}\Subset\Omega$; $R\leq\min\lbrace 1,\frac{\diam(X)}{18\lambda}\rbrace$. Let $u\in N^{1,1}_{\textrm{loc}}(\Omega)$ with $H(\cdot, g_u)\in L^1_{\textrm{loc}}(\Omega)$ be non-negative, $u \in DG_H(\Omega)$. Fix $\delta\in (0,1)$ and suppose that
\begin{equation}\label{eq2harnack}
\frac{\mu(B_R\cap\lbrace u>l\rbrace)}{\mu(B_R)}\geq \delta
\end{equation} 
holds for some positive level $l>0$. Then
\begin{equation}\label{eq3harnack}
\essinf_{B_R} u\geq \frac{l}{C_{1,\delta}},
\end{equation}
for a constant $C_{1,\delta}=C_{1,\delta}(\data,\Vert u\Vert_{L^{\infty}(B_{3\lambda R})},C_1,\delta)$.
\end{lemma}

\begin{proof}
The proof splits into two steps.

\textbf{Step 1: Density estimate.} We show that \eqref{eq2harnack} implies

\begin{equation}\label{eq3.5harnack}
\frac{\mu(B_R\cap\lbrace u\leq 2^{-\hat{j}}l\rbrace)}{\mu(B_R)}\leq C\hat{j}^{-\xi}
\end{equation}
for any $\hat{j}\in\mathbb{N}$, with $\xi=\xi(p,s)>0$ and $C_{1,\delta}=C_{1,\delta}(\data,\Vert u\Vert_{L^{\infty}(B_{3\lambda R})}, C_1,\delta)$.
For this, we follow the idea of the proof of Lemma \ref{lemma3frozen}.

As before, we define
\begin{equation}\label{eq4harnack}
v_j=\max\lbrace u, 2^{-j}l\rbrace-\max\lbrace u, 2^{-(j+1)}l\rbrace.
\end{equation}

Note that, by \eqref{eq4harnack}, we have
\begin{equation} \label{eq5harnack}
	v=\begin{cases}
		2^{-j}l-2^{-(j+1)}l=2^{-(j+1)}l, \hspace{2.15cm} \mbox{if $u\leq 2^{-(j+1)}l<2^{-j}l$,}\\
		2^{-j}l-u, \hspace{4.9cm} \mbox{if $2^{-(j+1)}l<u<2^{-j}l$,}\\
		0, \hspace{6cm} \mbox{if $u\geq 2^{-j}l>2^{-(j+1)}l$.}
	\end{cases}
\end{equation}
From \eqref{eq2harnack} and \eqref{eq5harnack}, we deduce that

\begin{align*}
\mu(\lbrace x\in B_R:v_j(x)>0\rbrace)&\leq \mu(\lbrace x\in B_R:u(x)\leq 2^{-j}l\rbrace)\leq(1-\delta)\mu(B_R),
\end{align*}
where $0<1-\delta<1$.

Therefore, by Lemma 2.1 of \cite{KS}, H\"older's inequality and the almost standard Caccioppoli inequality, Lemma \ref{Almost standard Caccioppoli's inequality}, we obtain
\begin{align*}
2^{-(j+1)}l\ \mu(B_R&\cap\lbrace u\leq 2^{-(j+1)}l\rbrace)=\int_{B_R\cap\lbrace u\leq 2^{-(j+1)}l\rbrace}v_j\dd\mu\leq \int_{B_R}v_j\dd\mu\\
&\leq\mu(B_R)^{1-\frac{1}{s}}\left(\int_{B_R}v_j^{s}\dd\mu\right)^{\frac{1}{s}}\leq C R\mu(B_R)^{1-\frac{1}{s}}\left(\int_{B_{\lambda R}}g_{v_j}^{s}\dd\mu\right)^{\frac{1}{s}}\\
&\leq C R \mu(B_R)^{1-\frac{1}{s}}\left(\int_{B_{\lambda R}\cap\lbrace 2^{-(j+1)}l<u\leq 2^{-j}l\rbrace}g_{v_j}^{s}\dd\mu\right)^{\frac{1}{s}}\\
&\leq C R \mu(B_R)^{1-\frac{1}{s}}\mu(B_{\lambda R}\cap\lbrace 2^{-(j+1)}l<u\leq 2^{-j}l\rbrace)^{\frac{1}{s}-\frac{1}{p}}\left(\int_{B_{\lambda R}\cap\lbrace 2^{-(j+1)}l<u\leq 2^{-j}l\rbrace}g_{v_j}^{p}\dd\mu\right)^{\frac{1}{p}}\\
&\leq C R \mu(B_R)^{1-\frac{1}{s}}\mu(B_{\lambda R}\cap\lbrace 2^{-(j+1)}l<u\leq 2^{-j}l\rbrace)^{\frac{1}{s}-\frac{1}{p}}\left(\int_{B_{\lambda R}\cap\lbrace 2^{-(j+1)}l<u\leq 2^{-j}l\rbrace}g_{(u-2^{-j}l)_-}^{p}\dd\mu\right)^{\frac{1}{p}}\\
&\leq C R \mu(B_R)^{1-\frac{1}{s}}\mu(B_{\lambda R}\cap\lbrace 2^{-(j+1)}l<u\leq 2^{-j}l\rbrace)^{\frac{1}{s}-\frac{1}{p}}\left(\int_{B_{\lambda R}}g_{(u-2^{-j}l)_-}^{p}\dd\mu\right)^{\frac{1}{p}}\\
&\leq C R \mu(B_R)^{1-\frac{1}{s}}\mu(B_{\lambda R}\cap\lbrace 2^{-(j+1)}l<u\leq 2^{-j}l\rbrace)^{\frac{1}{s}-\frac{1}{p}}\left(\int_{B_{2\lambda R}}\left\vert\frac{(u-2^{-j}l)_-}{3\lambda R}\right\vert^{p}\dd\mu\right)^{\frac{1}{p}}\\
&\leq C\mu(B_R)^{1-\frac{1}{s}}\mu(B_{\lambda R}\cap\lbrace 2^{-(j+1)}l<u\leq 2^{-j}l\rbrace)^{\frac{1}{s}-\frac{1}{p}}\left(\int_{B_{2\lambda R}}\left\vert(u-2^{-j}l)_-\right\vert^{p}\dd\mu\right)^{\frac{1}{p}}\\
&\leq C\mu(B_R)^{1-\frac{1}{s}+\frac{1}{p}}\mu(B_{\lambda R}\cap\lbrace 2^{-(j+1)}l<u\leq 2^{-j}l\rbrace)^{\frac{1}{s}-\frac{1}{p}}2^{-j}l.
\end{align*}
Therefore,
$$
\mu(B_{R}\cap\lbrace u\leq 2^{-(j+1)}l\rbrace)\leq C\mu(B_R)^{1-\frac{1}{s}+\frac{1}{p}}\mu(B_{\lambda R}\cap\lbrace 2^{-(j+1)}l<u\leq 2^{-j}l\rbrace)^{\frac{1}{s}-\frac{1}{p}}.
$$
If $\hat{j}>j$, then $\mu(B_{R}\cap\lbrace u\leq 2^{-\hat{j}}l\rbrace)\leq\mu(B_{R}\cap\lbrace u\leq 2^{-(j+1)}l\rbrace)$. Therefore,
$$
\mu(B_{R}\cap\lbrace u\leq 2^{-\hat{j}}l\rbrace)^{\frac{ps}{p-s}}\leq C\mu(B_R)^{\frac{ps-p-s}{p-s}}\mu(B_{\lambda R}\cap\lbrace 2^{-(j+1)}l<u\leq 2^{-j}l\rbrace).
$$
Summing up for $j=0,1,...,\hat{j}-1$, $\hat{j}\in\mathbb{N}$, and the doubling property of the measure, gives
\begin{align*}
\hat{j}\mu(B_{R}\cap\lbrace u\leq 2^{-\hat{j}}l\rbrace)^{\frac{ps}{p-s}}&\leq C\mu(B_R)^{\frac{ps-p-s}{p-s}}\sum_{j=0}^{\hat{j}-1}\mu(B_{\lambda R}\cap\lbrace 2^{-(j+1)}l<u\leq 2^{-j}l\rbrace)\\
&\leq  C\mu(B_R)^{\frac{ps-p-s}{p-s}+1}.
\end{align*}
Therefore,
$$
\frac{\mu(B_{R}\cap\lbrace u\leq 2^{-\hat{j}}l\rbrace)}{\mu(B_R)}\leq C\hat{j}^{-\left(\frac{p-s}{sp}\right)}=C\hat{j}^{-\xi},
$$
where $\xi=\xi(p,s)>0$ and $=C(\data,\Vert u\Vert_{L^{\infty}(B_{3\lambda R})},C_1,\delta).$

\textbf{Step 2: (Almost everywhere) pointwise strict positivity.} For this step, we follow the idea of the proof of Lemma \ref{lemma4frozen}.

Let $\hat{j}\in\mathbb{N}$ to be determined, and for $i\in\mathbb{N}$, we consider the radii and the levels
$$
r_i=R\left(\frac{1}{2}+\frac{1}{2^{i+1}}\right)\searrow\frac{R}{2},\qquad \tilde{r_i}=\frac{r_{i+1}+r_i}{2},\qquad k_i=\left(\frac{1}{2}+\frac{1}{2^{i+1}}\right)2^{-\hat{j}}l\searrow 2^{-\hat{j}-1}l
$$
Using the almost standard Caccioppoli inequality, Lemma \ref{Almost standard Caccioppoli's inequality} with $t=r_{i+1}$, $s=r_i$ and $k=k_i$ yields, since $\frac{R}{2}\leq r_i\leq R$ for all $i$
\begin{align}\label{eq6harnack}
\int_{B_{\tilde{r_i}}}g_{(u-k_i)_-}^{p}\dd\mu&\leq C\left(\frac{R}{r_i-\tilde{r_i}}\right)^{q}\int_{B_{r_i}}\left\vert\frac{(u-k_i)_-}{R}\right\vert^{p}\dd\mu\nonumber\\
&\leq C(2^{i+2})^{q}\int_{B_{r_i}}\left\vert\frac{(u-k_i)_-}{r_i}\right\vert^{p}\dd\mu\nonumber\\
&\leq C2^{qi}\int_{B_{r_i}}\left\vert\frac{(u-k_i)_-}{r_i}\right\vert^{p}\dd\mu,
\end{align}
where $C=C(C_1,\Vert u\Vert_{L^{\infty}(B_R)},[a]_{\alpha},p,q).$

Now, we take Lipschitz cut-off functions $\eta_i\in \textrm{Lip}_0(B_{\tilde{r_i}})$, $0\leq \eta_i\leq 1$, such that $\eta_i= 1$ on $B_{\tilde{r_i}}$ and $g_{\eta_i}\leq\frac{C2^{i}}{R}$.
By Theorem \ref{sstars} (see also \cite{BB}) our space supports a $(p^*,p)$-Poincar\'e inequality for $p^*=p^*(p,Q)>p$. Furthermore, by definition of $\eta_i$

\begin{equation}\label{eq7harnack}
\left(\dashint_{B_{r_{i+1}}}(u-k_i)_-^{p^*}\dd\mu\right)^{\frac{p}{p^*}}=\left(\dashint_{B_{r_{i+1}}}(\eta_i(u-k_i)_-)^{p^*}\dd\mu\right)^{\frac{p}{p^*}}
\end{equation}

On the other hand, by the previous step, and the doubling property of the measure, we get
\begin{align*}
\mu(\lbrace x\in B_{r_{i+1}}: \eta_i(u-k_i)_->0\rbrace)&\leq \mu(\lbrace x\in B_{r_{i+1}}: u\leq 2^{-\hat{j}}l\rbrace)\\
&\leq C\mu(B_R)\hat{j}^{-\xi}\leq C\mu(B_{R/2})\hat{j}^{-\xi}\\
&\leq C\mu(B_{r_{i+1}})\hat{j}^{-\xi}.
\end{align*}
Therefore, we focus on $\hat{j}\in\mathbb{N}$ big enough so that $C\hat{j}^{-\xi}<\frac{1}{2}$. Moreover, by Leibniz' rule we have

$$
g_{\eta_i(u-k_i)_-}\leq g_{(u-k_i)_-}\eta_i+(u-k_i)_-g_{\eta_i}\leq  g_{(u-k_i)_-}+\frac{C2^{i}}{R}(u-k_i)_-.
$$

So, by Lemma 2.1 of \cite{KS} with $t=p^*$ and $q=p$ and \eqref{eq6harnack}, we obtain

\begin{align*}
\left(\dashint_{B_{r_{i+1}}}(\eta_i(u-k_i)_-)^{p^*}\dd\mu\right)^{\frac{p}{p^*}}&\leq C r_{i+1}^{p}\dashint_{B_{\lambda r_{i+1}}}g_{\eta_i(u-k_i)_-}\dd\mu\\
&\leq Cr_{i+1}^{p}\mu(B_{r_{i+1}})^{-1}\int_{B_{\tilde{r_i}}}\left(g_{(u-k_i)_-}+\frac{C2^{i}}{R}(u-k_i)_-\right)^{p}\dd\mu\\
&\leq C 2^{iq}\dashint_{B_{\tilde{r_i}}}(u-k_i)_-^{p}\dd\mu.
\end{align*}

Therefore, by \eqref{eq7harnack} and this last set of inequalities
\begin{equation}\label{eq8harnack}
\left(\dashint_{B_{r_{i+1}}}(u-k_i)_-^{p^*}\dd\mu\right)^{\frac{p}{p^*}}\leq C 2^{iq}\dashint_{B_{\tilde{r_i}}}(u-k_i)_-^{p}\dd\mu. 
\end{equation}
where $C=C(\data,C_1,\Vert u\Vert_{L^{\infty}(B_R)}).$
We now estimate from above the right-hand side of \eqref{eq16frozen}. Since $u$ is non-negative, then $(u-k_i)_-\leq\chi_{\lbrace u\leq k_i\rbrace} 2^{-\hat{j}}l$ in $B_R$. On the other hand, to estimate from below the left-hand side of \eqref{eq16frozen}, we use

$$
\chi_{\lbrace u\leq k_{i+1}\rbrace} 2^{-\hat{j}-i-2}l=\chi_{\lbrace u\leq k_{i+1}\rbrace}(k_i-k_{i+1})\leq \chi_{\lbrace u\leq k_{i+1}\rbrace}(u-k_i)_-\leq (u-k_i)_.
$$
Therefore, the last three displays and \eqref{frozeninduction2}, yield
$$
\dashint_{B_{r_{i+1}}}\chi_{\lbrace u\leq k_{i+1}\rbrace}\dd\mu\leq C\left(2^{\frac{(q+p)p^*}{p}}\right)^i\left(\dashint_{B_{r_{i}}}\chi_{\lbrace u\leq k_{i}\rbrace}\dd\mu\right)^{1+\left(\frac{p^*-p}{p}\right)}.
$$
Therefore, by defining $A_i=\dashint_{B_{r_{i}}}\chi_{\lbrace u\leq k_{i}\rbrace}\dd\mu$, we obtain
$$
A_{i+1}\leq C\left(2^{\frac{(q+p)p^*}{p}}\right)^iA_i^{1+\left(\frac{p^*-p}{p}\right)},
$$
with $C=C(\data,C_1,\Vert u\Vert_{L^{\infty}(B_R)})$.
 By Lemma 7.1 in \cite{G}, we have that $A_i\rightarrow 0$ when $i\rightarrow\infty$ provided
 
$$
A_0\leq C^{-\frac{p}{p^*-p}}(2^{\frac{(q+p)p^*}{p}})^{-\frac{p^2}{(p^*-p)^{2}}}=\tilde{C}^{-1}=\tilde{C}^{-1}(\data,C_1,\Vert u\Vert_{L^{\infty}(B_R)}).
$$ 
So, if 
$$
\frac{\mu(\lbrace x\in B_R: u\leq 2^{-\hat{j}}\rbrace)}{\mu(B_R)}\leq \tilde{C}^{-1},
$$
then, 
$$
\dashint_{B_{R/2}}\chi_{\lbrace u\leq 2^{-\hat{j}-1}\rbrace}\dd\mu=0
$$
Therefore, $\essinf_{B_{R/2}} u\geq\frac{l}{2^{\hat{j}+1}}$ as wanted.

By step 1, we have
$$
\frac{\mu(B_{R}\cap\lbrace u\leq 2^{-\hat{j}}l\rbrace)}{\mu(B_R)}\leq C\hat{j}^{-\xi},
$$
for all $\hat{j}\in\mathbb{N}$ and for some $\xi=\xi(p,s)>0$ and $C=C(\data,C_1,\Vert u\Vert_{L^{\infty}(B_R)})$.

By asking $\hat{j}$ to be large enough so that $C\hat{j}^{-\xi}\leq\min\lbrace\frac{1}{2},\tilde{C}^{-1}\rbrace$, we obtain
$$
\essinf_{B_{R/2}} u\geq\frac{l}{2^{\hat{j}+1}}=\frac{l}{C_{1,\delta}},
$$

where $C_{1,\delta}=C_{1,\delta}(\data,C_1,\Vert u\Vert_{L^{\infty}(B_R)})$.

\end{proof}

At this point, we are ready to state the following weak Harnack inequality theorem. The proof is based on methods originally developed by DiBenedetto-Trudinger in \cite{DT} which are applicable in a wide range of contexts, for example see the recent contributions by Nastasi-Pacchiano Camacho \cite{NP} and Kinnunen-Shanmugalingam \cite{KS} in the metric setting and Baroni-Colombo-Mingione \cite{BCM} in the euclidean case. The weak Harnack inequality result states that functions in the De Giorgi class (see Definition \ref{DGclass}) are locally bounded. We note that the obtained estimate is a basis for our study.

\begin{theorem}\label{Theorem 7.1KS} Assume that \eqref{eq1harnack} holds for the concentric ball $B_{3\lambda R}\Subset\Omega$; $R\leq\min\lbrace 1,\frac{\diam(X)}{18\lambda}\rbrace$. Let $u\in N^{1,1}_{\textrm{loc}}(\Omega)$ with $H(\cdot, g_u)\in L^1_{\textrm{loc}}(\Omega)$ be non-negative and $-u \in DG_H(\Omega)$. Then there exist two constants 
	$C$ and $t_-> 0$  such that
	\begin{equation*}
		\essinf_{B_{3R}} u \geq C \left(\dashint_{B_R}u^{t_-} \, d\mu\right)^{\frac{1}{t_-}}
	\end{equation*}
	for every $B_R$ with $B_{9\lambda R} \Subset \Omega$ for $0 < R \leq\min\lbrace 1,\frac{\diam(X)}{18\lambda}\rbrace$, such that \eqref{eq1harnack} holds. The
	constants $C$ and $\sigma$ are independent of the ball $B_R$.
\end{theorem}
\begin{proof} The proof follows exactly as Proof of Theorem 7.1 of \cite{KS}.
\end{proof}

\begin{theorem}\label{esssupharnack} Assume that \eqref{eq1harnack} holds for the concentric ball $B_{3\lambda R}\Subset\Omega$; $R\leq\min\lbrace 1,\frac{\diam(X)}{18\lambda}\rbrace$. Let $u\in DG_H(\Omega)$ and $-u\in DG_H(\Omega)$. Then, for any exponent  $t_+>0$ and radii $0<\rho<R$, the local estimate

$$
\esssup_{B_{\rho}} u_+\leq \frac{C}{(1-\frac{\rho}{R})^{\frac{Q}{t_+}}}\left(\dashint_{B_R}u^{t_+}\dd\mu\right)^{\frac{1}{t_+}}.
$$
holds, where $C=C(\data,C_1, \Vert u\Vert_{L^{\infty}(B_R)})$.

\end{theorem}
\begin{proof}
We follow the ideas of the proof of Theorem \ref{upperharnackineqfrozen}.
For $k>0$ to be chosen later and any positive integer $n$, we set
$$
k_n=k(1-2^{-n})\nearrow k,\qquad\qquad R_n=(1+2^{-n})R/2\searrow R/2.
$$
and 

$$
\bar{R_n}=\frac{1}{2}(R_n+R_{n+1}),\qquad\qquad \rho_n=\frac{R/2}{2^{n+2}}=R_n-\bar{R_n}.
$$
We grab cut-off functions $\eta_n\in \textrm{Lip}_0(B_{\bar{R_n}})$ with compact support in $B_{\bar{R_n}}$, $\eta_n= 1$ on $B_{R_{n+1}}$ and $0\leq \eta_n\leq 1$, $g_{eta_n}\leq\frac{2^{n+2}}{R/2}=\frac{1}{\rho_n}.$

By Theorem \ref{sstars}, there is a $p^{*}=p^*(p,Q)>p$ such that $X$ supports a $(p^*,p)$-Poincar\'e inequality. Furthermore, by definition
$$
\left(\dashint_{B_{r_{n+1}}}(u-k_n)_+^{p^*}\dd\mu\right)^{\frac{p}{p^*}}=\left(\dashint_{B_{r_{i+1}}}(\eta_n(u-k_n)_+)^{p^*}\dd\mu\right)^{\frac{p}{p^*}}
$$

Now, by Theorem 5.51 of \cite{BB}, we get
\begin{align*}
\left(\dashint_{B_{r_{i+1}}}(\eta_n(u-k_n)_+)^{p^*}\dd\mu\right)^{\frac{p}{p^*}}&\leq C\bar{R_n}^{p}\dashint_{B_{\bar{R_n}}}g_{\eta_n(u-k_n)_+}^{p}\dd\mu\\
&\leq CR_n^{p}\dashint_{B_{\bar{R_n}}}g_{\eta_n(u-k_n)_+}^{p}\dd\mu.
\end{align*}

Therefore,

\begin{equation}\label{eq9harnack}
\left(\dashint_{B_{r_{n+1}}}(u-k_n)_+^{p^*}\dd\mu\right)^{\frac{p}{p^*}}\leq CR_n^{p}\dashint_{B_{\bar{R_n}}}g_{\eta_n(u-k_n)_+}^{p}\dd\mu.
\end{equation}

On the other hand, by Leibniz' rule and the almost standard Caccioppoli inequality, Lemma \ref{Almost standard Caccioppoli's inequality}, we obtain
\begin{align*}
\dashint_{B_{\bar{R_n}}}g_{\eta_n(u-k_n)_+}^{p}\dd\mu&\leq C\left(\dashint_{B_{\bar{R_n}}}g_{(u-k_n)_+}^{p}\dd\mu+\dashint_{B_{\bar{R_n}}}\left(\frac{2^{n+2}(u-k_n)_+}{R/2}\right)^{p}\dd\mu\right)\\
&\leq C 2^{qn}\dashint_{B_{R_n}}\left\vert\frac{(u-k_n)_+}{R}\right\vert^{p}\dd\mu\leq C 2^{qn}\dashint_{B_{R_n}}\left\vert\frac{(u-k_n)_+}{R_n}\right\vert^{p}\dd\mu,
\end{align*}
where $C=C(\data, C_1, \Vert u\Vert_{L^{\infty}(B_R)})$.

Since the levels $\lbrace k_n\rbrace$ are increasing, notice that we have the following estimates
$$
(u-k_n)^{p^*}_+\geq (u-k_n)_+^{p^*-p}(u-k_{n+1})^{p^*}_+\geq(2^{-n-1}k)^{p^*-p}(u-k_{n+1})^{p^*}_+
$$
By \eqref{eq9harnack} and the last two inequalities, we have
$$
\left((2^{-n-1}k)^{p^*-p}\dashint_{B_{r_{n+1}}}(u-k_{n+1})_+^{p^*}\dd\mu\right)^{\frac{p}{p^*}}\leq CR_n^{p}\dashint_{B_{R_n}}(u-k_n)_+^{p}\dd\mu.
$$
So,
$$
k^{-p}\dashint_{B_{r_{n+1}}}(u-k_{n+1})_+^{p^*}\dd\mu\leq C \left(2^{(\frac{qp^*}{p}+p^*-p)}\right)^{n}\left(k^{-p}\dashint_{B_{R_n}}(u-k_n)_+^{p}\dd\mu\right)^{1+\frac{(p^*-p)}{p}},
$$
with $C=C(\data, C_1, \Vert u\Vert_{L^{\infty}(B_R)})$.

Now, denoting 
$$
Y_n=k^{-p}\dashint_{B_{R_n}}(u-k_n)_+^{p}\dd\mu.
$$
Then, the previous inequality rewrites as
$$
Y_{n+1}\leq \left(C\mu(B_R)^{\frac{1}{d_1}-1}R\right)(2^{\tilde{C}(q,p,Q)})^{n}Y_n^{1+\frac{p^*-p}{p}}.
$$
 By Lemma 7.1 in \cite{G}, we have that $Y_n\rightarrow 0$ provided
 $$
 Y_0\leq C^{-1}=C^{-1}(\data,C_1, \Vert u\Vert_{L^{\infty}(B_R)}),
 $$

Therefore, by choosing $k$, such that
$$
K\geq C\left(\dashint_{B_R}u_+^p\dd\mu\right)^{\frac{1}{p}},
$$ 
with $C=C(\data,C_1, \Vert u\Vert_{L^{\infty}(B_R)})$
So, we have
$$
\lim_{n\rightarrow\infty}Y_n=0
$$
Then,
$$
\left(\dashint_{B_{R/2}}(u-k)_+^p\dd\mu\right)^{\frac{1}{p}}=0
$$
Therefore, 
$$
\esssup_{B_{R/2}}u\leq K=\left(\dashint_{B_R}u_+^p\dd\mu\right)^{\frac{1}{p}}
$$
Finally, by the interpolation argument at the end of proof of Theorem \ref{upperharnackineqfrozen}, we obtain the result.
\end{proof}
At this point, by combining the estimates obtained both for the $p$-phase and the $(p,q)$-phase, we are ready to prove the Harnack inequality. 
The proof now follows easily by merging the previous results. 
\begin{proof}[Proof of Theorem \ref{Corollary 7.3 KS}]
Let $B_R\Subset \Omega$ be a ball with $0<R<\min\lbrace 1, \frac{\diam(X)}{18\lambda}\rbrace$. If \eqref{eq1harnack} holds, then the proof follows combining Theorem \ref{Theorem 7.1KS} and Theorem \ref{esssupharnack}.
If, on the contrary, \eqref{eq1harnack} fails, then we can conclude by Theorem \ref{upperharnackineqfrozen} and Theorem \ref{essinfharnackfrozen}.
    
\end{proof}

\section*{Acknowledgements}

A. Nastasi is a member of the Gruppo Nazionale per l'Analisi Matematica, la Probabilit\`{a} e le loro Applicazioni (GNAMPA) of the Istituto Nazionale di Alta Matematica (INdAM).
A. Nastasi was partly supported by  GNAMPA-INdAM  Project 2023 "Regolarità per problemi ellittici e parabolici con crescite non standard", CUP E53C22001930001.

\section*{Declarations}
\subsection*{Conflict of interest} The authors declare that they have no conflict of interest.

\end{document}